%% file: runsmaptree_final.tex
\documentclass[11pt,reqno,a4paper]{amsart}
\usepackage{amsfonts,amsmath,amssymb}
\usepackage{a4wide}
\usepackage{graphicx}
\usepackage{colonequals}
\usepackage{tikz}
\usetikzlibrary{arrows,shapes,patterns,automata,calc}
\usepackage{bm}

\makeatletter
\tikzset{circle split part fill/.style  args={#1,#2}{%
 alias=tmp@name, 
  postaction={%
    insert path={
     \pgfextra{%
     \pgfpointdiff{\pgfpointanchor{\pgf@node@name}{center}}%
                  {\pgfpointanchor{\pgf@node@name}{east}}%
     \pgfmathsetmacro\insiderad{\pgf@x}
      \fill[#1] (\pgf@node@name.base) ([xshift=-\pgflinewidth]\pgf@node@name.east) arc
                          (0:180:\insiderad-\pgflinewidth)--cycle;
      \fill[#2] (\pgf@node@name.base) ([xshift=\pgflinewidth]\pgf@node@name.west)  arc
                           (180:360:\insiderad-\pgflinewidth)--cycle;            
         }}}}}  
\makeatother

\allowdisplaybreaks

\newtheorem{theorem}{Theorem}

\theoremstyle{definition}
\newtheorem{remark}{Remark}
\newtheorem{example}{Example}

\newcommand{\Stir}[2]{\genfrac{ \{ }{ \} }{0pt}{}{#1}{#2}}

\newcommand{\Euler}[2]{\genfrac{ \langle }{ \rangle }{0pt}{}{#1}{#2}}

\newcommand\xqed[1]{%
  \leavevmode\unskip\penalty9999 \hbox{}\nobreak\hfill
  \quad\hbox{#1}}
\newcommand\demo{\xqed{$\dashv$}}

\title{Runs in labelled trees and mappings}

\author[M.-L.~Lackner]{Marie-Louise Lackner}
\address{Marie-Louise Lackner\\
Forschungsbereich Databases and Artificial Intelligence\\
Technische Universit\"at Wien\\
Favoritenstra{\ss}e 9-11, 1040 Wien, Austria} %
\email{marie-louise.lackner@tuwien.ac.at}

\author[A.~Panholzer]{Alois Panholzer}
\address{Alois Panholzer\\
Institut f{\"u}r Diskrete Mathematik und Geometrie\\
Technische Universit\"at Wien\\
Wiedner Hauptstra{\ss}e 8-10/104\\
1040 Wien, Austria} \email{Alois.Panholzer@tuwien.ac.at}

\date{\today}


\thanks{Both authors were supported by the Austrian Science Fund FWF, grant P25337-N23.}

\begin{document}

\begin{abstract} We generalize the concept of ascending and descending runs from permutations to rooted labelled trees and mappings, i.e., functions from the set $\{1, \dots, n\}$ into itself.
A combinatorial decomposition of the corresponding functional digraph together with a generating functions approach allows us to perform a joint study of ascending and descending runs in labelled trees and mappings, respectively. From the given characterization of the respective generating functions we can deduce bivariate central limit theorems for these quantities.
Furthermore, for ascending runs (or descending runs) we gain explicit enumeration formul{\ae} showing a connection to Stirling numbers of the second kind. We also give a bijective proof establishing this relation, and further state a bijection between mappings and labelled trees connecting the quantities in both structures. 
\end{abstract}

\maketitle

\section{Introduction}\label{sec:Introduction}
An $n$-mapping is a function $f : [n] \to [n]$ from the set of integers $[n] \colonequals\left\lbrace 1, 2, \ldots, n\right\rbrace$ into itself. Random $n$-mappings, i.e., where one of these $n^{n}$ functions is chosen with equal probability, appear in various applications, e.g., in cryptography and for occupancy problems.
For an $n$-mapping $f$ the functional digraph (also called mapping graph) $G_{f} = (V,E)$ is the directed graph with vertex-set $V=[n]$ and edge-set $E=\{(i,f(i)) : i \in [n]\}$. Structural properties of the functional digraphs of random mappings have widely been studied, see, e.g., the work of Arney and Bender~\cite{ArnBen1982}, Kolchin~\cite{Kol1986}, Flajolet and Odlyzko~\cite{FlaOdl1990}.
For instance, it is well known that the expected number of connected components in a random $n$-mapping is asymptotically $1/2 \cdot  \log(n)$, the expected number of cyclic nodes is $\sqrt{\pi n/2}$, and the expected number of terminal nodes, i.e., nodes with no preimages, is $e^{-1}n$.

In the functional digraph $G_{f}$ corresponding to a random mapping $f$ the nodes' labels play an important r\^{o}le and thus it is somewhat surprising that so far there are only few studies concerning occurrences of label patterns.
One such label pattern are ascending edges\footnote{Actually, they are also called ``ascents'' in the literature but due to a use notion of this term by Gessel in~\cite{Ges1996} to which we will refer later in this work, we use ``ascending edges'' instead.}, i.e., edges $e=(x,y)$ in $G_{f}$ with $x<y$ (note that throughout this work we will always identify a node with its label). By providing a family of weight preserving bijections, E\u{g}ecio\u{g}lu and Remmel~\cite{EgeRem1986} showed how results on ascending edges in mappings (which are amenable rather easily) can be translated to corresponding ones in Cayley trees, i.e., rooted labelled trees. Using their results, Clark~\cite{Cla2008} provided central limit theorems for the number of ascending edges in random mappings and in random rooted labelled trees.
Another research direction concerned with patterns formed by the labels in a functional digraph can be found in \cite{Pan2013} where alternating mappings have been studied; these are a generalization of the concept of alternating permutations to mappings.
They can be defined as mappings for which every iteration orbit $i, f(i), f^{2}(i), \dots$ forms an alternating sequence, or alternatively for which the mapping graph $G_{f}$ neither contains two consecutive ascending edges nor two consecutive descending edges. Results for mappings could be obtained by using and extending corresponding results on labelled tree families, so-called alternating trees~\cite{KubPan2010}.
In this context we also want to mention the PhD thesis of Okoth~\cite{Oko2015}, who  studied local extrema in trees (called sources and sinks there); his studies also led to results for the corresponding quantities in mappings.

In this work we analyse fundamental label patterns in random mappings by generalizing the notion of ascending and descending runs from permutations to mappings. When considering the mapping graph $G_{f}$ of a mapping $f$, ascending runs are maximal ascending paths, thus corresponding to iteration sequences $i < f(i) < \dots < f^{r}(i)$ that are not contained in a larger such sequence; analogous for descending runs. We carry out a joint study of the number of ascending and descending runs in mappings by characterizing the generating function of the number of mappings of a certain size with a prescribed number of ascending and descending runs, and analyse the typical behaviour of these label patterns in random mappings by characterizing the limiting distribution behaviour. When restricting the analysis to a single pattern (i.e., either ascending runs or descending runs), one even gets explicit enumeration formul{\ae} via Stirling numbers of the second kind for which we also provide a bijective proof.

In our generating functions approach for an analysis of runs in mappings we also performed a study of the corresponding quantities in rooted labelled trees, and such results for labelled trees might be of interest in their own right. Interestingly, the enumeration formul{\ae} for labelled trees and mappings, respectively, of given size with a prescribed number of ascending runs (or descending runs) are closely related, for which we can also give a bijective argument.

In the combinatorial literature various studies of quantities related to the labelling of trees can be found. Besides the work already mentioned, e.g., there are studies for labelled trees (or forests) concerning the size of the maximal subtree without ascending edges containing the root node \cite{SeoShi2012}, proper vertices \cite{GesSeo2006} (also called leaders, i.e., nodes $x$ with largest label in the whole subtree rooted at $x$), and proper edges \cite{Sho1995} (edges $e=(y,x)$, with $x$ closer to the root, where $x$ has a label larger than all nodes in the subtree rooted at $y$). We further want to mention the very recent work \cite{AndArch2019} showing enumerative results for forests avoiding certain sets of subsequence patterns. Concerning the present studies, the work \cite{Ges1996} of Gessel is of particular interest, where so-called descent and leaves in forests of rooted labelled trees are considered jointly. There, a descent is defined as a node, which has at least one child with a larger label. It turns out that the generating functions presented in Gessel's work are closely related to the ones obtained in our studies of ascending and descending runs in trees. As a consequence, distributional results obtained here can easily be transferred to those quantities. Moreover, an explicit enumeration result for the number of ascending runs in labelled trees can already be obtained from \cite{Ges1996}.

We want to point out that the generating functions approach presented in this work relying on a decomposition of the structures w.r.t.\ the smallest (or largest) labelled element is flexible enough to obtain results for further kind of label patterns and other tree families as well. In particular, as preliminary results show, some questions raised in \cite{AndArch2019} concerning avoidance and occurrence of consecutive patterns of length $3$ in forests of rooted trees could be treated; we will comment on that elsewhere.

The paper is organized as follows: In Section~\ref{sec:PreliminariesResults}, after preliminary comments, we state the main results of this work concerning generating functions, and exact and limiting distribution results for the number of ascending and descending runs in Cayley trees and mappings. In Section~\ref{sec:GF} we carry out the generating functions approach for a joint study of the quantities considered. Exact enumeration results for ascending runs are deduced in Section~\ref{sec:AscendingRuns}, but the main part of this section is devoted to a bijective proof of these result and establishing the correspondence between the number of ascending runs in mappings and trees. Section~\ref{sec:JointStudies} shows a bivariate central limit theorem for the number of ascending and descending runs in labelled trees and mappings. Moreover, we use relations to generating functions occurring in \cite{Ges1996} to prove a bivariate central limit theorem for ascents and leaves in trees.

\section{Preliminaries and main results}\label{sec:PreliminariesResults}

\subsection{Preliminaries}\label{ssec:Preliminaries}

In our studies we use the close relation between mapping graphs of functions and rooted labelled trees, i.e., Cayley trees. In this work, when we speak about a labelled tree, we always mean a rooted unordered tree (i.e., there is no ordering on the subtrees of any node), where every node in a tree of size $n$ carries a distinct integer from the set $[n]$ as a label. As mentioned earlier, a node and its label in trees or mapping graphs are used synonymously. In accordance with the connection to mapping graphs we consider the edges in the tree as oriented towards the root node. Thus, throughout this work, instead of using the terms children or parent of a node, we speak about in-neighbours and out-neighbour, respectively.
The number $T_{n}$ of labelled trees of size $n$, where size is always measured by the number of nodes, is given by $T_{n} = n^{n-1}$, a formula attributed to Arthur Cayley. When speaking about a random tree of size $n$, one of these $T_{n}$ trees is chosen with equal probability. The exponential generating function $T(z) \colonequals \sum_{n \geq 1} T_{n} \frac{z^{n}}{n!}$ of labelled trees, the so-called tree function, is characterized via the functional equation
\begin{equation}\label{eqn:TreeFunction}
  T(z)=z e^{T(z)},
\end{equation}
and is thus closely related to the Lambert $W$-function \cite{CorGonHarJefKnu1996}.

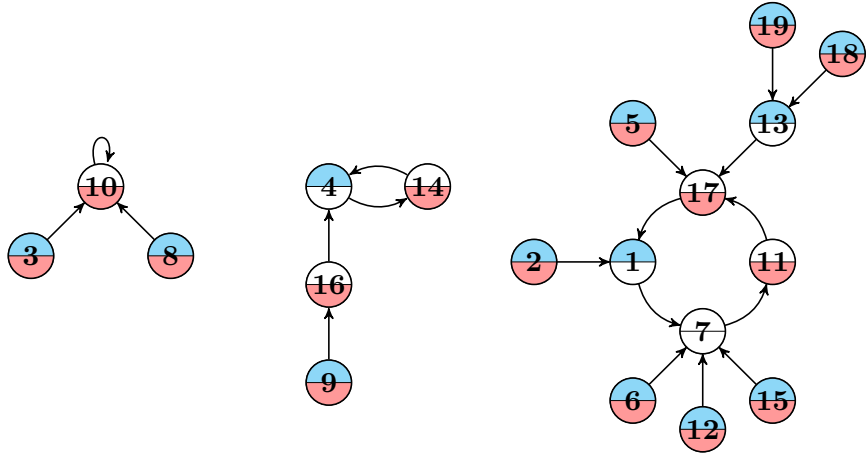
\begin{figure}
\centering
\input{FunctionalGraphGrey}
\caption{Mapping graph of a $19$-mapping with three connected components, which consist of one, two and four Cayley trees, respectively. Starting nodes of ascending runs are coloured black (filling the upper semicircle), starting nodes of descending runs are coloured grey (filling the lower semicircle). Thus this mapping has $13$ ascending runs and $15$ descending runs.}
\label{Mappings/fig:mapping_graph}
\end{figure}

The structure of mapping graphs is simple and is well described in \cite{FlaSed2009}: the weakly connected components of such graphs are just cycles of Cayley trees.  That is, each connected component consists of rooted labelled trees whose root nodes are connected by directed edges such that they form a cycle. For an example of the functional digraph of a $19$-mapping, see Figure~\ref{Mappings/fig:mapping_graph}.
Using the symbolic method (see~\cite{FlaSed2009} for an introduction), this structural connection between Cayley trees and mappings can also easily be taken to the level of generating functions and yields the relation $M(z) = 1/(1-T(z))$, with $M(z) \colonequals \sum_{n \ge 0} M_{n} \frac{z^{n}}{n!}$ the exponential generating function of the number $M_{n} = n^{n}$ of $n$-mappings. For the problem considered here, this relation between labelled trees and mappings cannot be applied directly, but we will rather use a decomposition of the objects with respect to the node with smallest label. Such a decomposition takes care of the quantities studied, but yields more involved relations leading to linear or quasi-linear first-order partial differential equations (PDEs) for the corresponding generating functions.

Considering a mapping graph or a labelled tree, an ascending run (or descending run) is a maximal directed path $x_{1}, x_{2}, \dots, x_{r}$ of nodes in the graph forming an ascending (or descending) sequence of labels, $x_{1} < x_{2} < \cdots < x_{r}$ (or $x_{1} > x_{2} > \cdots > x_{r}$). Crucial to our approach is a simple characterization of the starting node of an ascending run: The node $x$ is the starting node of an ascending run exactly if $x$ doesn't have an in-neighbour with a smaller label. Analogously, $x$ is the starting node of a descending run iff there is no in-neighbour of $x$ with a larger label. For mappings, one could also say that $x$ doesn't have a smaller (or larger) preimage. In Figure~\ref{Mappings/fig:mapping_graph} the starting nodes of ascending and descending runs, respectively, are coloured black and grey.

Throughout this work we use $X \stackrel{(d)}{=} Y$ to denote equality in distribution of random variables (r.v.\ for short) $X$ and $Y$, whereas $X_{n} \xrightarrow{(d)} X$ means weak convergence, i.e., convergence in distribution, of the sequence of r.v.\ $X_{n}$ to the r.v.\ $X$. $\mathcal{N}(\mu, \sigma^{2})$ denotes the normal distribution with mean $\mu$ and variance $\sigma^{2}$, and $\mathcal{N}(\bm{\mu},\bm{\Sigma})$ a two-dimensional normal distribution with mean vector $\bm{\mu} \in \mathbb{R}^{2}$ and variance-covariance matrix $\bm{\Sigma} \in \mathbb{R}^{2 \times 2}$. Furthermore, we use $x^{\underline{k}} \colonequals x \cdot (x-1) \cdots (x-k+1)$ for the falling factorials and $\Stir{n}{m}$ for the Stirling numbers of the second kind, i.e, the number of partitions of a set of $n$ labelled objects into $k$ nonempty unlabelled subsets.

\subsection{Results}

\subsubsection{Ascending runs}

\begin{theorem}\label{thm:FnmGnm}
Let $\hat{F}_{n,m}$ be the number of rooted labelled trees with $n$ nodes and $m$ ascending runs, and $\hat{G}_{n,m}$ be the number of $n$-mappings with $m$ ascending runs. Then $\hat{F}_{n,m}$ and $\hat{G}_{n,m}$ are given as follows:
\begin{equation*}
  \hat{F}_{n,m} = (n-1)^{\underline{m-1}} \cdot \Stir{n}{m}, \qquad	
	\hat{G}_{n,m} = n \hat{F}_{n,m} = n^{\underline{m}} \cdot \Stir{n}{m}.
\end{equation*}
\end{theorem}

\begin{theorem}\label{thm:XnaYna}
Let $X_{n}^{[a]}$ and $Y_{n}^{[a]}$ be the random variables counting the number of ascending runs in a random size-$n$ rooted labelled tree and a random $n$-mapping, respectively. Then $X_{n}^{[a]} \stackrel{(d)}{=} Y_{n}^{[a]}$ holds for $n \ge 1$, and expectation and variance are given as follows:
\begin{align*}
  \mathbb{E}(X_{n}^{[a]}) & = \mathbb{E}(Y_{n}^{[a]}) = n \cdot (1-(1-n^{-1})^{n}) = (1-e^{-1}) n + \mathcal{O}(1) \sim 0.632120\ldots \cdot n,\\
	\mathbb{V}(X_{n}^{[a]}) & = \mathbb{V}(Y_{n}^{[a]}) = (e^{-1}-2e^{-2}) n + \mathcal{O}(1) \sim 0.097208\ldots \cdot n.
\end{align*}
Moreover, the normalized r.v.\ $\tilde{X}_{n}^{[a]} \colonequals \frac{X_{n}^{[a]} - \mathbb{E}(X_{n}^{[a]})}{\sqrt{\mathbb{V}(X_{n}^{[a]})}}$ and $\tilde{Y}_{n}^{[a]} \colonequals \frac{Y_{n}^{[a]} - \mathbb{E}(Y_{n}^{[a]})}{\sqrt{\mathbb{V}(Y_{n}^{[a]})}}$ converge in distribution to a standard normal distributed r.v., $\tilde{X}_{n}^{[a]} \stackrel{(d)}{=} \tilde{Y}_{n}^{[a]} \xrightarrow{(d)} \mathcal{N}(0,1)$.
\end{theorem}

\begin{remark}
We may compare the results for the number of ascending runs in labelled trees and mappings with corresponding ones in permutations. Whereas the number of labelled trees and mappings of size $n$ with $m$ ascending runs is related to the Stirling numbers of the second kind, the number $a_{n,m}$ of permutations of $[n]$ with $m$ ascending runs is given by the (shifted) Eulerian numbers (see, e.g.,~\cite{GraKnuPat1994}), $a_{n,m} = \Euler{n}{m-1}$, where $\Euler{n}{m}$ counts the number of permutations of $[n]$ with $m$ ascents (or $m$ descents), i.e., elements in the permutation larger than the preceding one. With $R_{n}$ the r.v.\ counting the number of ascending runs in a randomly chosen permutation of $[n]$, one gets mean $\mathbb{E}(R_{n}) = \frac{n+1}{2}$ and variance $\mathbb{V}(R_{n}) = \frac{n+1}{12}$. As in labelled trees and mappings, the number of runs in permutations converges, after normalization, in distribution to a standard normal distribution (see~\cite{CarKurScoStack1972}): $\frac{R_{n} - \mathbb{E}(R_{n})}{\sqrt{\mathbb{V}(R_{n})}} \xrightarrow{(d)} \mathcal{N}(0,1)$, but the coefficients occurring in the leading asymptotics of mean and variance differ from the ones in trees and mappings.
\end{remark}

\subsubsection{Joint study of ascending and descending runs}

\begin{theorem}\label{thm:FzvwGzvw}
Let $F_{n,m,\ell}$ be the number of rooted labelled trees with $n$ nodes, $m$ ascending and $\ell$ descending runs, $G_{n,m,\ell}$ the number of $n$-mappings with $m$ ascending and $\ell$ descending runs, and $F(z,v,w) \colonequals \sum\limits_{n \ge 1}\sum\limits_{m \ge 0} \sum\limits_{\ell \ge 0} F_{n,m,\ell} \frac{z^{n} v^{m} w^{\ell}}{n!}$ and $G(z,v,w) \colonequals \sum\limits_{n \ge 0}\sum\limits_{m \ge 0} \sum\limits_{\ell \ge 0} G_{n,m,\ell} \frac{z^{n} v^{m} w^{\ell}}{n!}$ their generating functions. Then $F \colonequals F(z,v,w)$ is characterized as the solution of the functional equation
\begin{large}
\begin{equation*}
  z = \frac{\ln\left(\frac{(e^{F}-1+v)(e^{F}-1+w)}{vwe^{F}}\right)}{e^{F}-(1-v)(1-w)},
\end{equation*}
\end{large}
and $G \colonequals G(z,v,w)$ is given via
\begin{large}
\begin{equation*}
\begin{textstyle}
  G = \frac{e^{F} \left( e^{F}-(1-v)(1-w)\right)^{2}}{\left( e^{F}-(1-v)(1-w)\right)\left( e^{2F}-(1-v)(1-w)\right)-e^{F}(e^{F}-1+v)(e^{F}-1+w)\ln\left(\frac{(e^{F}-1+v)(e^{F}-1+w)}{vwe^{F}}\right)}.
\end{textstyle}
\end{equation*}
\end{large}
\end{theorem}

\begin{theorem}\label{thm:XnYn}
Let $\bm{X}_{n} \colonequals \left(\begin{smallmatrix}X_{n}^{[a]}\\ X_{n}^{[d]}\end{smallmatrix}\right)$ be the random vector counting the number of ascending $X_{n}^{[a]}$ and descending runs $X_{n}^{[d]}$ in a random size-$n$ rooted labelled tree, and $\bm{Y}_{n} \colonequals \left(\begin{smallmatrix}Y_{n}^{[a]}\\ Y_{n}^{[d]}\end{smallmatrix}\right)$ the corresponding random vector in a random $n$-mapping. 
Then, after suitable normalization, $\bm{X}_{n}$ and $\bm{Y}_{n}$, respectively, converge in distribution to a two-dimensional normal distribution $\mathcal{N}(\bm{\mu},\bm{\Sigma})$ with mean vector $\bm{\mu} = \bm{0} \colonequals \left(\begin{smallmatrix}0\\ 0\end{smallmatrix}\right)$,
\begin{equation*}
  \frac{1}{\sqrt{n}} \left(\bm{X}_{n} - \left(\begin{smallmatrix}1-e^{-1}\\ 1-e^{-1}\end{smallmatrix}\right) \cdot n\right) \xrightarrow{(d)} \mathcal{N}(\bm{0}, \bm{\Sigma}), \qquad
	\frac{1}{\sqrt{n}} \left(\bm{Y}_{n} - \left(\begin{smallmatrix}1-e^{-1}\\ 1-e^{-1}\end{smallmatrix}\right) \cdot n\right) \xrightarrow{(d)} \mathcal{N}(\bm{0}, \bm{\Sigma}),
\end{equation*}
and where the variance-covariance matrix $\bm{\Sigma}$ is given as follows:
\begin{equation*}
\begin{textstyle}
  \bm{\Sigma} = \left(\begin{array}{cc}e^{-1}-2e^{-2} & e^{-1}-3e^{-2}\\e^{-1}-3e^{-2} & e^{-1}-2e^{-2}\end{array}\right)
	= \left(\begin{array}{rr}0.097208\ldots & -0.038126\ldots\\-0.038126\ldots & 0.097208\ldots\end{array}\right).
\end{textstyle}
\end{equation*}
\end{theorem}

\section{Generating functions approach}\label{sec:GF}

\subsection{Runs in labelled trees}\label{ssec:GF_RunsTrees}

In order to perform a joint study of ascending and descending runs in labelled trees we will think about trees where each node is coloured black if it is the starting node of a maximal ascending run and  grey if it is the starting node of a maximal descending run. Since nodes can be coloured black and grey simultaneously, which happens exactly for leaves, we might think that each node contains two buttons, a black button and a grey button, which could be pressed or not. Recall that for any labelled tree a node $x$ is the starting node of a maximal ascending run (and thus coloured black), iff each in-neighbour of $x$ has a label larger than $x$, whereas it is the starting node of a maximal descending run (and thus coloured grey), iff each in-neighbour of $x$ has a label smaller than $x$. Let us now introduce the combinatorial family $\mathcal{F}$ of labelled trees with nodes coloured as described before.

\begin{figure}
\begin{center}
\includegraphics[height=4cm]{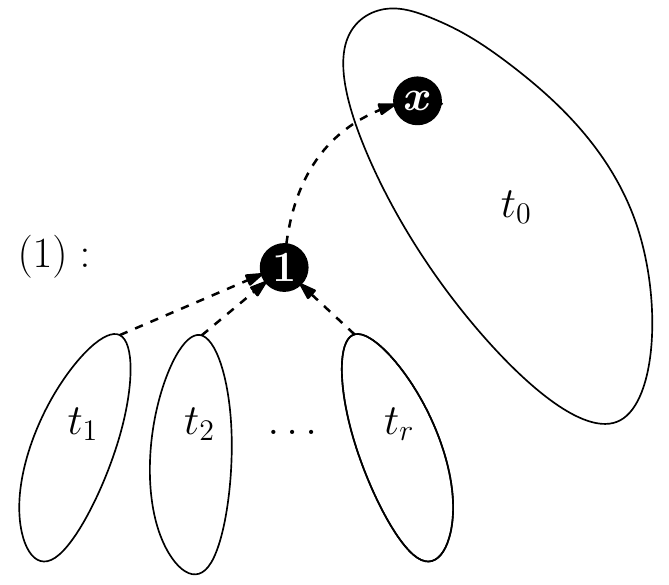} \enspace
\includegraphics[height=4cm]{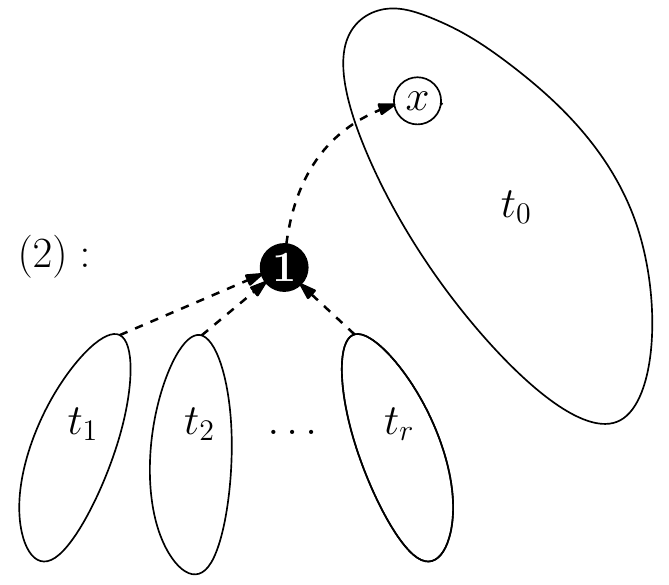} \enspace
\includegraphics[height=4cm]{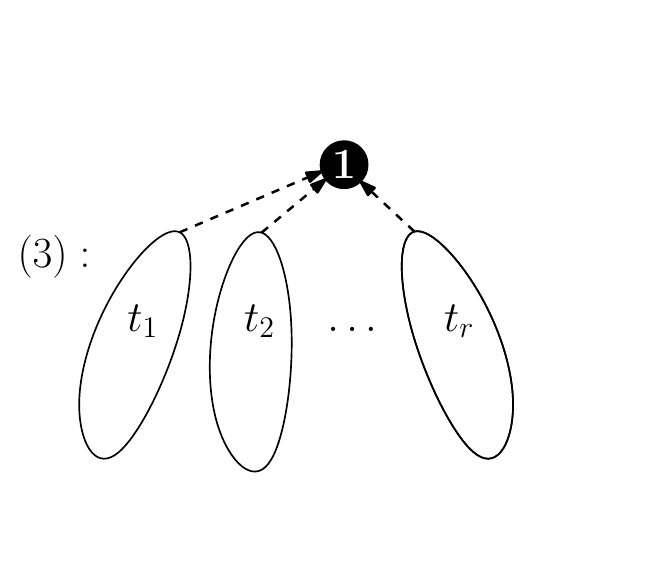}
\end{center}
\caption{Decomposition of a tree $t$ with respect to the node with smallest label, where the three cases described below may occur.}
\label{fig:TreeDecomposition}
\end{figure}

Our approach is based on the decomposition of a labelled tree $t$ w.r.t.\ the node with smallest label $1$ into the node $1$, $r \ge 0$ subtrees $t_{1}, \dots, t_{r}$ attached to $1$ and, if $1$ is not the root of $t$, a subtree $t_{0}$, where node $1$ is attached to one of its nodes $x$. Note that the node $1$ is always the starting node of an ascending run and thus coloured black. Furthermore the node $1$ is in $t$ coloured grey iff it does not have in-neighbours, thus $r=0$ in the above decomposition. Three cases can occur (see Figure~\ref{fig:TreeDecomposition}): 
\begin{itemize}
\item[$(1)$] Node $1$ is not the root of $t$ and $x$ is a black node in $t_{0}$: then the total number of black nodes in $t$ is the sum of the number of black nodes in $t_{0}, t_{1}, \dots t_{r}$, since in $t$ node $x$ loses the black colour, whereas the black node $1$ is added.
\item[$(2)$] Node $1$ is not the root of $t$ and $x$ is not black in $t_{0}$: then the total number of black nodes in $t$ is one (for node $1$) plus the sum of the number of black nodes in $t_{0}, t_{1}, \dots t_{r}$, since the colour of $x$ in $t$ remains unchanged.
\item[$(3)$] Node $1$ is the root of $t$: then the total number of black nodes in $t$ is one (for node $1$) plus the sum of the number of black nodes in $t_{1}, \dots t_{r}$.
\end{itemize}
Each of these cases can be divided into two subcases $(a)$ and $(b)$ depending on whether $r>0$ or $r=0$:
\begin{itemize}
\item[$(a)$] $r>0$: Then the total number of grey nodes in $t$ is the sum of the number of grey nodes in the subtrees $t_{0}$ (if occurring), $t_{1}, \dots t_{r}$.
\item[$(b)$] $r=0$: Then the total number of grey nodes in $t$ is one (for node $1$) plus the sum of the number of grey nodes in the subtree $t_{0}$.
\end{itemize}

To gain a symbolic equation for $\mathcal{F}$ using this decomposition we use basic combinatorial constructions (see \cite{FlaSed2009}) as the disjoint union $+$, the partition product $\ast$ and the set-construction $\text{\textsc{Set}}$ of labelled families; furthermore  $\mathcal{A}^{\Box} \ast \mathcal{B}$ denotes the boxed-product of the families $\mathcal{A}$ and $\mathcal{B}$, where the smallest label $1$ has to be contained in the $\mathcal{A}$-component. With $\mathcal{Z}$ we denote an atomic element, i.e., a vertex, and with $\epsilon$ an empty structure. Furthermore we use marking-operators: $\Theta_{\mathcal{Z}}(\mathcal{A})$ contains all structures obtained by distinguishing (i.e., marking) one node in an object of $\mathcal{A}$; to mark a black vertex or a grey vertex we use the markers $B$ and $Y$, respectively, and $\Theta_{B}(\mathcal{A})$ contains all structures obtained by distinguishing a black node in an object of $\mathcal{A}$. With these constructions the above decomposition can be described formally as follows, where the summands in the formal equation correspond to the cases occurring:
\begin{equation}\label{eqn:F_SymbolicEqn}
\begin{split}
  \mathcal{F} & = \mathcal{Z}^{\Box} \ast \Theta_{B}(\mathcal{F}) \ast \left( \text{\textsc{Set}}(\mathcal{F})\setminus\{\epsilon\} + \{\epsilon\} \times \{Y\}\right) \\
	& \quad \mbox{} + \mathcal{Z}^{\Box} \ast \left( \Theta_{\mathcal{Z}}(\mathcal{F}) \setminus \Theta_{B}(\mathcal{F})\right) \ast \left( \text{\textsc{Set}}(\mathcal{F})\setminus\{\epsilon\} + \{\epsilon\} \times \{Y\}\right) \times \{B\} \\
	& \quad \mbox{} + \mathcal{Z}^{\Box} \ast \left( \text{\textsc{Set}}(\mathcal{F})\setminus\{\epsilon\} + \{\epsilon\} \times \{Y\}\right) \times \{B\}.
\end{split}
\end{equation}

We introduce the trivariate generating function
\begin{align*}
  F(z,v,w) & \colonequals \sum_{t \in \mathcal{F}} \frac{z^{|t|} \, v^{\text{$\sharp$ black nodes in $t$}} \, w^{\text{$\sharp$ grey nodes in $t$}}}{|t|!} = \sum_{n \ge 1} \sum_{m \ge 0} \sum_{\ell \ge 0} F_{n,m,\ell} \, \frac{z^{n} v^{m} w^{\ell}}{n!} \\
	& = \sum_{n \ge 1} \sum_{m \ge 0} \sum_{\ell \ge 0} n^{n-1} \mathbb{P}\big\{X_{n}^{[a]} = m \text{ and } X_{n}^{[d]}=\ell\big\} \frac{z^{n} v^{m} w^{\ell}}{n!},
\end{align*}
where $F_{n,m,\ell}$ denotes the number of labelled trees of size $n$ with $m$ ascending runs and $\ell$ descending runs, and the r.v.\ $X_{n}^{[a]}$ and $X_{n}^{[d]}$ count the number of ascending runs and descending runs, respectively, in a random labelled tree of size $n$. Then, by applying the symbolic method, the formal equation \eqref{eqn:F_SymbolicEqn} yields the following first-order quasilinear PDE for $F \colonequals F(z,v,w)$:
\begin{equation*}
  F_{z} = vF_{v} (e^{F}-1+w) + v(zF_{z}-vF_{v})(e^{F}-1+w) + v(e^{F}-1+w),
\end{equation*}
with initial condition $F(0,v,w)=0$. Note that the boxed-product $\mathcal{C} = \mathcal{A}^{\Box} \ast \mathcal{B}$ 
yields the equation $C_{z} = A_{z} \cdot B$ at the level of generating functions. Moreover, since the marking operators $\Theta_{\mathcal{Z}}$ and $\Theta_{B}$ applied to $\mathcal{F}$ generate $n F_{n,m,\ell}$ and $m F_{n,m,\ell}$ different trees of size $n$ with $m$ ascending and $\ell$ descending runs, respectively, this leads to expressions $z F_{z}$ and $v F_{v}$ in the above equation.
This PDE can be rewritten as follows:
\begin{equation}\label{eqn:Fzvw_PDE}
  \left( 1-vz(e^{F}-1+w)\right)F_{z} - v(1-v)(e^{F}-1+w)F_{v}-v(e^{F}-1+w)=0,
\end{equation}

The solution of \eqref{eqn:Fzvw_PDE} can be obtained by a standard application of the method of characteristics for first-order quasilinear PDEs (see, e.g., \cite{Eva2010}). We give a sketch of the computations, since for the corresponding study of runs in mappings we require a first integral occurring here. Introducing a function $f=f(z,v,F)$ and assuming $f(z,v,F(z,v))=\text{const.}$ (we consider $w$ as a parameter), we obtain after taking partial derivatives the following PDE for $f$:
\begin{equation*}
  \left( 1-vz(e^{F}-1+w)\right)f_{z} -v(1-v)(e^{F}-1+w)f_{v}+v(e^{F}-1+w)f_{F}=0.
\end{equation*}
To find solutions of the PDE we consider the system of characteristic equations (by assuming that the variables occurring are dependent on a parameter $t$, $z=z(t)$, $v=v(t)$, $F=F(t)$, and using the notation $\dot{z} = \frac{dz}{dt}$, etc.):
\begin{equation}\label{eqn:F_CharacteristicEquations}
  \dot{z} = 1-vz(e^{F}-1+w), \quad 	\dot{v} = -v(1-v)(e^{F}-1+w), \quad \dot{F} = v(e^{F}-1+w).
\end{equation}
From the second and the third characteristic equation \eqref{eqn:F_CharacteristicEquations} we easily get the first integral $\frac{e^{F}}{1-v} = C_{1} = \text{const.}$. By using this result, the first and the second characteristic equation \eqref{eqn:F_CharacteristicEquations} yield, by solving a first-order linear ordinary differential equation, the first integral $ze^{F}-\frac{e^{F}}{e^{F}-(1-v)(1-w)} \ln(\frac{e^{F}-1+w}{v}) = C_{2} = \text{const.}$ Combining them, we deduce that the general solution of \eqref{eqn:Fzvw_PDE} satisfies
\begin{equation*}
  z e^{F} - \frac{e^{F}}{e^{F}-(1-v)(1-w)} \ln\left( \frac{e^{F}-1+w}{v}\right) = h\left( \frac{e^{F}}{1-v}\right) ,
\end{equation*}
with a certain differentiable function $h(x)$. By taking into account the initial condition $F(0,v,w)=0$, we obtain the characterization $h(x) = \frac{x}{1-x-w} \ln(\frac{wx}{x-1})$, which shows that $F(z,v,w)$ is indeed solution of the functional equation stated in Theorem~\ref{thm:FzvwGzvw}:
\begin{equation}\label{eqn:Fzvw_Solution}
  z = \frac{\ln\left(\frac{(e^{F}-1+v)(e^{F}-1+w)}{vwe^{F}}\right)}{e^{F}-(1-v)(1-w)}.
\end{equation}

\subsection{Runs in mappings}

In order to study ascending and descending runs in $n$-mappings it suffices to consider the weakly connected components, so-called connected mappings. Combinatorially, mappings and connected mappings are linked by the $\text{\textsc{Set}}$-construction and thus we can easily transfer results from one family to the other. We start by considering connected mappings for which, analogously to our previous analysis of labelled trees, each node is coloured black if it is the starting node of a maximal ascending run and grey if it is the starting node of a maximal descending run. Note that for a connected mapping there could occur a cycle of length one leading to an in-neighbour with the same label; therefore we use the characterization given in Section~\ref{ssec:Preliminaries}, i.e., a node $x$ is the starting node of an ascending run iff there is no in-neighbour of $x$ with a smaller label; analogous for descending runs. We introduce the combinatorial family $\mathcal{C}$ of connected mappings with nodes coloured as described before.

\begin{figure}
\begin{center}
\includegraphics[height=4cm]{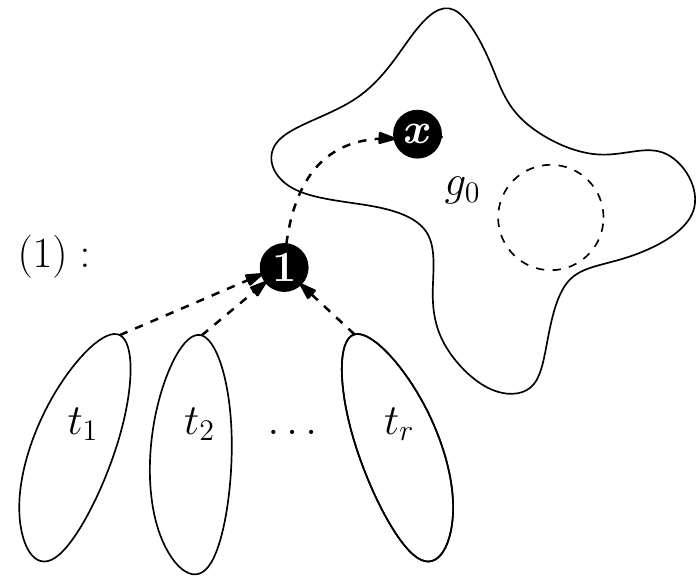} \enspace
\includegraphics[height=4cm]{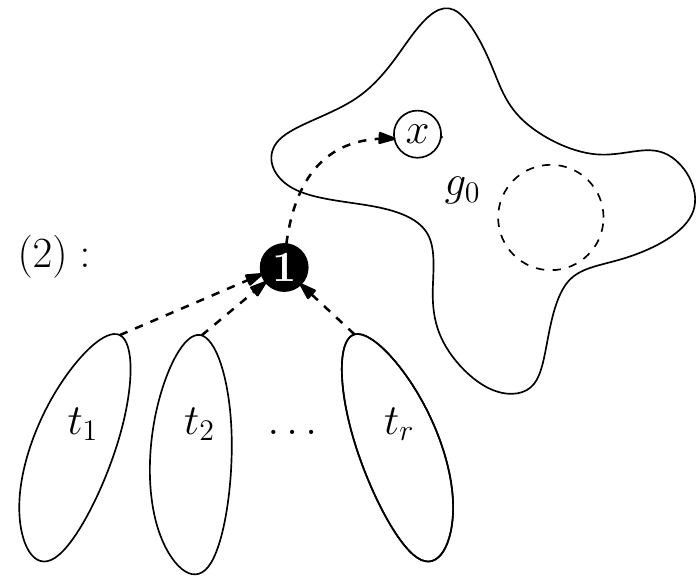} \enspace
\includegraphics[height=4cm]{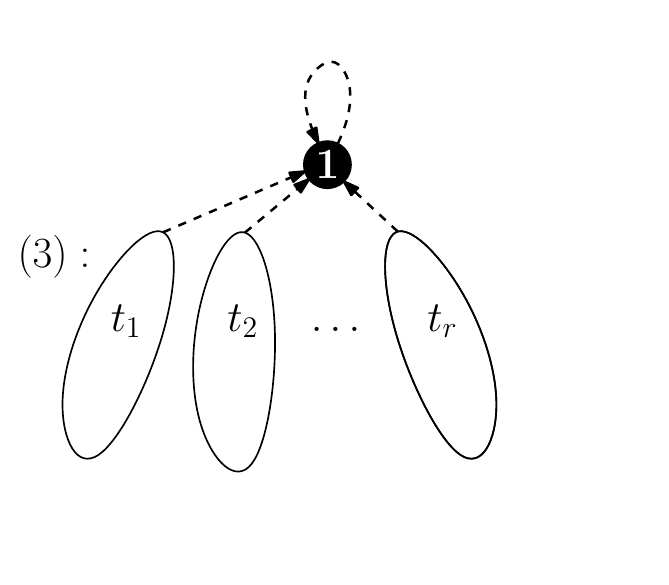}\newline
\includegraphics[height=4cm]{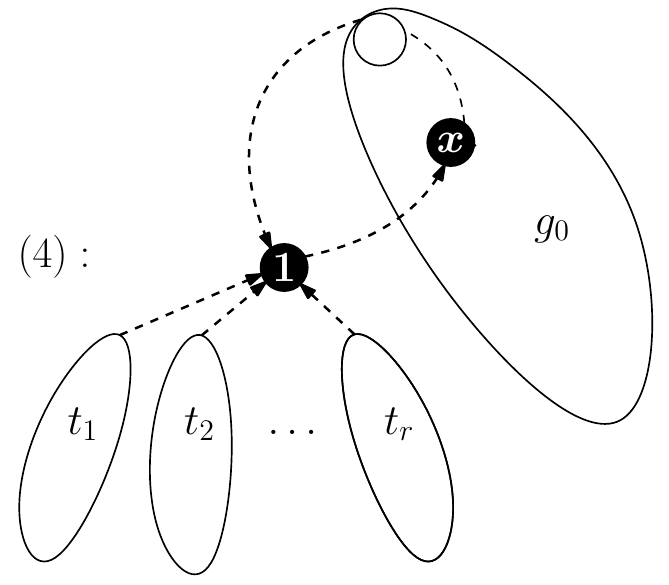} \enspace
\includegraphics[height=4cm]{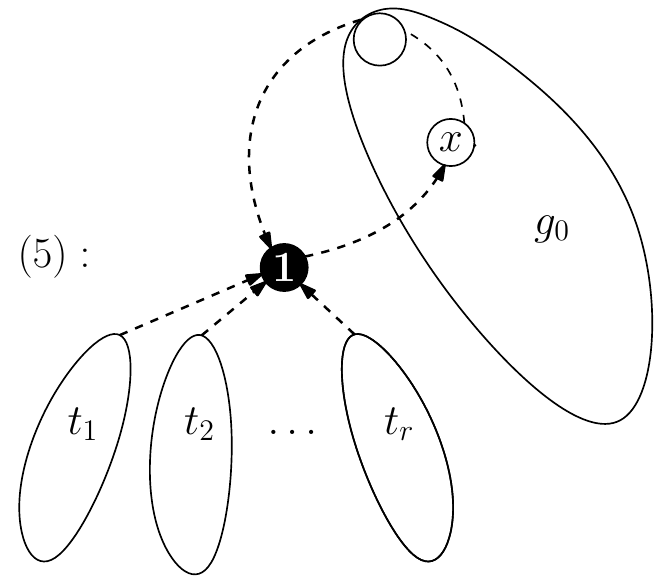}
\end{center}
\caption{Decomposition of a connected mapping $c$ with respect to the node with smallest label, where the five cases described below may occur.}
\label{fig:MappingDecomposition}
\end{figure}

Our approach relies on the decomposition of a connected mapping $c$ w.r.t.\ the node with smallest label $1$ into the node $1$, $r \ge 0$ subtrees $t_{1}, \dots, t_{r}$ attached to $1$ and, if $1$ is not part of a loop, a structure $g_{0}$, where the node $1$ is attached to one of its nodes $x$. Depending on whether $1$ is contained in a cycle or not, $g_{0}$ itself is a subtree or a connected mapping. The node $1$ is always the starting node of an ascending run and thus coloured black. Furthermore the node $1$ is coloured grey iff it does not have in-neighbours (others than $1$, if $1$ is part of a loop). Five cases can occur (see Figure~\ref{fig:MappingDecomposition}): 
\begin{itemize}
\item[$(1)$] Node $1$ is not contained in a cycle of $c$ and $x$ is a black node in the connected mapping $g_{0}$: then the total number of black nodes in $c$ is the sum of the number of black nodes in $g_{0}, t_{1}, \dots t_{r}$, since in $c$ the node $x$ loses the black colour, whereas the black node $1$ is added.
\item[$(2)$] Node $1$ is not contained in a cycle of $c$ and $x$ is not black in the connected mapping $g_{0}$: then the total number of black nodes in $c$ is one (for node $1$) plus the sum of the number of black nodes in $g_{0}, t_{1}, \dots t_{r}$, since the colour of $x$ in $c$ remains unchanged.
\item[$(3)$] Node $1$ is part of a loop in $c$: then the total number of black nodes in $c$ is one (for node $1$) plus the sum of the number of black nodes in $t_{1}, \dots t_{r}$.
\item[$(4)$] Node $1$ is part of a non-loop cycle of $c$ and $x$ is a black node in the tree $g_{0}$: then the total number of black nodes in $c$ is the sum of the number of black nodes in $g_{0}, t_{1}, \dots t_{r}$, since in $c$ the node $x$ loses the black colour, whereas the black node $1$ is added.
\item[$(5)$] Node $1$ is part of a non-loop cycle of $c$ and $x$ is not black in the tree $g_{0}$: then the total number of black nodes in $c$ is one (for node $1$) plus the sum of the number of black nodes in $g_{0}, t_{1}, \dots t_{r}$, since the colour of $x$ in $c$ remains unchanged.
\end{itemize}
Each of the cases $(1)$--$(3)$ can be divided into two subcases $(a)$ and $(b)$ depending on whether $r>0$ or $r=0$:
\begin{itemize}
\item[$(a)$] $r>0$: Then the total number of grey nodes in $c$ is the sum of the number of grey nodes in the substructures $g_{0}$ (if occurring), $t_{1}, \dots t_{r}$.
\item[$(b)$] $r=0$: Then the total number of grey nodes in $c$ is one (for node $1$) plus the sum of the number of grey nodes in the substructure $g_{0}$.
\end{itemize}

Note that in cases $(4)$ and $(5)$ node $1$ cannot be the starting node of a descending run, thus no further distinction into cases occurs. Using the corresponding family $\mathcal{F}$ of coloured labelled trees introduced in Section~\ref{ssec:GF_RunsTrees} as well as the constructions described there, we can easily translate this decomposition into the following symbolic equation for the family of coloured connected mappings $\mathcal{C}$ (again, the summands in this formal description correspond to the previously described cases):
\begin{align}\label{eqn:C_SymbolicEqn}
  \mathcal{C} & = \mathcal{Z}^{\Box} \ast \Theta_{B}(\mathcal{C}) \ast \left( \text{\textsc{Set}}(\mathcal{F})\setminus\{\epsilon\} + \{\epsilon\} \times \{Y\}\right) \notag \\
	& \quad \mbox{} + \mathcal{Z}^{\Box} \ast \left( \Theta_{\mathcal{Z}}(\mathcal{C}) \setminus \Theta_{B}(\mathcal{C})\right) \ast \left( \text{\textsc{Set}}(\mathcal{F})\setminus\{\epsilon\} + \{\epsilon\} \times \{Y\}\right) \times \{B\} \notag \\
	& \quad \mbox{} + \mathcal{Z}^{\Box} \ast \left( \text{\textsc{Set}}(\mathcal{F})\setminus\{\epsilon\} + \{\epsilon\} \times \{Y\}\right) \times \{B\} \\
	& \quad \mbox{} + \mathcal{Z}^{\Box} \ast \Theta_{B}(\mathcal{F}) \ast \text{\textsc{Set}}(\mathcal{F}) \notag \\
	& \quad \mbox{} + \mathcal{Z}^{\Box} \ast \left( \Theta_{\mathcal{Z}}(\mathcal{F}) \setminus \Theta_{B}(\mathcal{F})\right) \ast \text{\textsc{Set}}(\mathcal{F}) \times \{B\}. \notag
\end{align}

We introduce the trivariate generating function
\begin{equation*}
  C(z,v,w) \colonequals \sum_{c \in \mathcal{C}} \frac{z^{|c|} \, v^{\text{$\sharp$ black nodes in $c$}} \, w^{\text{$\sharp$ grey nodes in $c$}}}{|c|!} = \sum_{n \ge 1} \sum_{m \ge 0} \sum_{\ell \ge 0} C_{n,m,\ell} \, \frac{z^{n} v^{m} w^{\ell}}{n!},
\end{equation*}
where $C_{n,m,\ell}$ denotes the number of connected $n$-mappings with $m$ ascending runs and $\ell$ descending runs. An application of the symbolic methods to the formal equation \eqref{eqn:C_SymbolicEqn} leads to the following first-order linear PDE for $C \colonequals C(z,v,w)$, with $F=F(z,v,w)$ the corresponding generating function for trees studied in Section~\ref{ssec:GF_RunsTrees}:
\begin{equation*}
  C_{z} = vC_{v} (e^{F}-1+w) + v(zC_{z}-vC_{v})(e^{F}-1+w) + v(e^{F}-1+w) + vF_{v} e^{F} + v(zF_{z}-vF_{v})e^{F}.
\end{equation*}
Taking into account \eqref{eqn:Fzvw_PDE}, slight simplifications occur yielding the following PDE together with the initial condition $C(0,v,w)=0$:
\begin{equation}\label{eqn:Czvw_PDE}
  \left( 1-vz(e^{F}-1+w)\right)C_{z} - v(1-v)(e^{F}-1+w)C_{v} = (1+v(1-w)z) F_{z} + v(1-v)(1-w) F_{v}.
\end{equation}

In order to solve this PDE we search for a suitable substitution of variables, such that it can be reduced to an ordinary differential equation; since the coefficients of the partial derivatives in the defining equations \eqref{eqn:Fzvw_PDE} and \eqref{eqn:Czvw_PDE} of the functions $F$ and $C$, respectively, match, this suggests to choose a first integral obtained for $F$. Furthermore, since the function $F$ is given only implicitly via the functional equation \eqref{eqn:Fzvw_Solution}, it is slightly tricky to get well tractable expressions, but it turns out that the following pair of substitutions works fine (where we consider $w$ as a parameter):
\begin{equation*}
  H = H(z,v) \colonequals \ln\left(\frac{(e^{F}-1+v)(e^{F}-1+w)}{vwe^{F}}\right), \qquad K = K(z,v) \colonequals \frac{e^{F}}{1-v}.
\end{equation*}
Namely, the inverse transform for $z=z(H,K)$ and $v=v(H,K)$ is given by
\begin{equation*}
  z = \frac{H K (e^{H}w+K-1)}{(K-1+w)(e^{H}Kw-Kw+K+w-1)}, \qquad v = \frac{(K-1)(K-1+w)}{K(e^{H}w+K-1)},
\end{equation*}
and by introducing $\tilde{C}(H,K) \colonequals C\left( z(H,K),v(H,K)\right)$ we get from \eqref{eqn:Czvw_PDE}, after some computations (done best with the help of a computer algebra system), the equation
\begin{equation*}
\begin{textstyle}
  \frac{\partial}{\partial H}\tilde{C}(H,K)=\frac{e^{H}(1-K)\left( e^{H}K^{2}w+(HK^{2}+K^{2}-HK-2K-w+1)(1-w)\right)w}{(e^{H}Kw-Kw+K+w-1)\left( e^{H}(HK^{2}-K^{2}-HK-w+1)w+(K-1)^{2}(w-1)\right)}.
\end{textstyle}
\end{equation*}
Integrating w.r.t.\ $H$ and adapting to the initial condition leads to the following solution of $\tilde{C}=\tilde{C}(H,K)$:
\begin{equation*}
  \tilde{C}=\ln\left(\frac{(e^{H}Kw-Kw+K+w-1)(1-w-K)}{e^{H} (HK^{2}-K^{2}-HK-w+1)w+(K-1)^{2}(w-1)}\right),
\end{equation*}
and backsubstitution gives the solution of $C(z,v,w)$, which is here omitted. Instead, we are interested in results for arbitrary (not only connected) mappings and thus introduce the trivariate generating function
\begin{equation*}
  G(z,v,w) \colonequals \sum_{n \ge 0} \sum_{m \ge 0} \sum_{\ell \ge 0} G_{n,m,\ell} \, \frac{z^{n} v^{m} w^{\ell}}{n!}
	= \sum_{n \ge 0} \sum_{m \ge 0} \sum_{\ell \ge 0} n^{n} \mathbb{P}\big\{Y_{n}^{[a]} = m \text{ and } Y_{n}^{[d]}=\ell\big\} \frac{z^{n} v^{m} w^{\ell}}{n!},
\end{equation*}
where $G_{n,m,\ell}$ denotes the number of $n$-mappings with $m$ ascending runs and $\ell$ descending runs, and the r.v.\ $Y_{n}^{[a]}$ and $Y_{n}^{[d]}$ count the number of ascending runs and descending runs, respectively, in a random $n$-mapping. Due to the $\text{\textsc{Set}}$-construction leading from connected mappings to mappings, it simply holds that $G=e^{C}$ for the respective generating functions, and the above result for $\tilde{C}(H,K)$ leads to the following solution of $G=G(z,v,w)$ stated in Theorem~\ref{thm:FzvwGzvw}, with $F=F(z,v,w)$ the corresponding generating function for labelled trees:
\begin{equation}\label{eqn:Gzvw_Solution}
\begin{textstyle}
  G = \frac{e^{F} \left( e^{F}-(1-v)(1-w)\right)^{2}}{\left( e^{F}-(1-v)(1-w)\right)\left( e^{2F}-(1-v)(1-w)\right)-e^{F}(e^{F}-1+v)(e^{F}-1+w)\ln\left(\frac{(e^{F}-1+v)(e^{F}-1+w)}{vwe^{F}}\right)}.
\end{textstyle}
\end{equation}
We remark that by differentiating \eqref{eqn:Fzvw_Solution} w.r.t.\ $z$ and comparing with \eqref{eqn:Gzvw_Solution} one further obtains the connection
\begin{equation}\label{eqn:FzvwGzvw_Connection}
  e^{F} F_{z} = (e^{F}-1+v) (e^{F}-1+w) G.
\end{equation}

\section{Ascending runs}\label{sec:AscendingRuns}

\subsection{Exact enumeration}

In this section, we consider ascending runs (without taking into account descending runs) in labelled trees and mappings, for which we can provide exact enumeration results and combinatorial explanations via bijections. Due to symmetry arguments all results also hold for a single study of descending runs. Let $\hat{F}_{n,m}$ and $\hat{G}_{n,m}$ be the number of labelled trees of size $n$ and $n$-mappings, respectively, with $m$ ascending runs, and $\hat{F}(z,v) \colonequals \sum_{n \ge 1}\sum_{m \ge 0} \hat{F}_{n,m} \frac{z^{n} v^{m}}{n!}$ and 
$\hat{G}(z,v) \colonequals \sum_{n \ge 0}\sum_{m \ge 0} \hat{G}_{n,m} \frac{z^{n} v^{m}}{n!}$ the corresponding generating functions. Clearly, it holds that $\hat{F}(z,v) = F(z,v,1)$ and $\hat{G}(z,v) = G(z,v,1)$, with $F$ and $G$ the trivariate generating functions studied in Section~\ref{sec:GF}.
This gives the characterization of $\hat{F} \colonequals \hat{F}(z,v)$ via the functional equation
\begin{equation}\label{eqn:Fzv_Solution}
  z = \frac{\ln\left(\frac{e^{\hat{F}}-1+v}{v}\right)}{e^{\hat{F}}},
\end{equation}
and of $\hat{G} \colonequals \hat{G}(z,v)$ by the relation
\begin{equation}\label{eqn:Gzv_Solution}
  \hat{G} = \frac{e^{\hat{F}}}{e^{\hat{F}}-(e^{\hat{F}}-1+v)\ln\left(\frac{e^{\hat{F}}-1+v}{v}\right)}.
\end{equation}

In order to extract coefficients it turns out to be advantageous to introduce the function $\hat{H} \colonequals \hat{H}(z,v) = \ln\left(\frac{e^{\hat{F}}-1+v}{v}\right)$. Simple manipulations show that $\hat{H}$ is characterized via the functional equation
\begin{equation*}
  z = \frac{\hat{H}}{v(e^{\hat{H}}-1)+1}.
\end{equation*}
We note that it follows from this characterization that $\hat{H} = \sum_{n \ge 1} \sum_{m \ge 0} \hat{H}_{n,m} \frac{z^{n}v^{m}}{n!}$ is the exponential generating function of the number $\hat{H}_{n,m}$ of labelled trees of size $n$ with exactly $m$ internal nodes (i.e., non-leaf nodes). Since $\hat{H}_{z} = \frac{\left( v(e^{\hat{H}}-1)+1\right)^{2}}{v(1-\hat{H})e^{\hat{H}}+1-v}$ and
\begin{equation*}
  \hat{G} = \frac{v(e^{\hat{H}}-1)+1}{v(1-\hat{H})e^{\hat{H}}+1-v},
\end{equation*}
as follows after easy computations, an application of Cauchy's integral formula (or alternatively, by taking formal residues) and taking into account the relation $[z^{n}](e^{z}-1)^{m} = \frac{m!}{n!} \Stir{n}{m}$ for the Stirling numbers of the second kind (see, e.g., \cite{GraKnuPat1994}), gives the explicit result for the coefficients $\hat{G}_{n,m}$ stated in Theorem~\ref{thm:FnmGnm}:
\begin{align}
  \hat{G}_{n,m} & = n! [z^{n} v^{m}] \hat{G}(z,v) = [v^{m}] \frac{n!}{2\pi i} \oint \frac{\hat{G}(z,v)}{z^{n+1}} dz \notag\\
	& = [v^{m}] \frac{n!}{2\pi i} \oint \frac{(v(e^{\hat{H}}-1)+1)^{n+1}}{\hat{H}^{n+1}} \cdot \frac{(v(e^{\hat{H}}-1)+1)}{v(1-\hat{H})e^{\hat{H}}+1-v} \cdot \frac{v(1-\hat{H})e^{\hat{H}}+1-v}{(v(e^{\hat{H}}-1)+1)^{2}} \, d\hat{H} \notag\\
	& = n! [\hat{H}^{n} v^{m}] \left( v(e^{\hat{H}}-1)+1\right)^{n}
	= n! \binom{n}{m} [\hat{H}^{n}] \left( e^{\hat{H}}-1\right)^{m} \notag\\
	& = \frac{n!}{(n-m)!} \Stir{n}{m} = n^{\underline{m}} \Stir{n}{m}. \label{eqn:Gnm_Formula}
\end{align}
The corresponding enumeration result for labelled trees stated in Theorem~\ref{thm:FnmGnm}, 
\begin{equation*}
  \hat{F}_{n,m} = (n-1)^{\underline{m-1}} \Stir{n}{m}, \quad \text{for $n \ge 1$},
\end{equation*}
could be obtained in a similar way by extracting coefficients, with $\hat{F} = \ln\left( v(e^{\hat{H}}-1)+1\right)$. However, taking the derivative of $\hat{F}$ given by \eqref{eqn:Fzv_Solution} w.r.t.\ $z$ easily shows that
\begin{equation*}
  \hat{G} = 1+z \hat{F}_{z},
\end{equation*}
which, at the level of coefficients, gives the following relation and thus also proves above enumeration result for $\hat{F}_{n,m}$:
\begin{equation}\label{eqn:GnmFnm_Relation}
  \hat{G}_{n,m} = n \hat{F}_{n,m}, \quad \text{for $n \ge 1$}.
\end{equation}
Bijective proofs of the explicit enumeration result for $\hat{G}_{n,m}$ and the connection between $\hat{G}_{n,m}$ and $\hat{F}_{n,m}$ are provided in the next subsection.

\subsection{Bijective proofs}

First, we give a bijective proof of the fact that the number $\hat{G}_{n,m}$ of $n$-mappings with $m$ ascending runs can be expressed with the help of the Stirling numbers of the second kind as previously stated.
The idea of the bijection is to successively decompose a mapping into ascending runs. 
This is done by starting with a run ending at the largest element of the mapping, then one ending at the next-largest element that has not been involved yet, and so on. 
The runs then correspond to blocks of the partition.
In order to keep track of how these runs were ``glued'' together and to be able to reconstruct the mapping, we additionally store the image of the last element of each run in the sequence $q$.

We shall prove the following:
\begin{theorem}\label{Mappings/thm:bij_runs_partitions}
 There is a bijection between the set of $n$-mappings with exactly $m$ runs and the set of pairs $(S,q)$, where $S$ is a set-partition of $[n]$ into $m$ parts and $q=(n_{1},\dots,n_{m})$ is an integer sequence of length $m$.
The set partition is given as $S=(S_{1},S_{2}, \dots, S_{m})$ where the parts are ordered decreasingly according to the largest element in each part, i.e., it holds $\max(S_{1}) > \max(S_{2}) > \cdots > \max(S_{m})$.
The sequence $q$ then has to satisfy the following restriction: $n_{j} \in [n] \setminus \left(\bigcup_{i=1}^{j-1} \min\{\ell \in S_{i} : \ell > \max(S_{j})\}\right)$.
\end{theorem}

\begin{figure}
\begin{minipage}[h]{0.45\linewidth}
\centering
\begin{tabular}{|c|c|} \hline
$S_1=\left\lbrace 19\right\rbrace$ & $n_1=13$ \\ \hline
$S_2=\left\lbrace 18\right\rbrace$ & $n_2=13$ \\ \hline
$S_3=\left\lbrace 17, 13\right\rbrace$ & $n_3=1$ \\ \hline
$S_4=\left\lbrace 16, 9\right\rbrace$ & $n_4=4$ \\ \hline
$S_5=\left\lbrace 15\right\rbrace$ & $n_5=7$ \\ \hline
$S_6=\left\lbrace 14, 4 \right\rbrace$ & $n_6=4$ \\ \hline
$S_7=\left\lbrace 12\right\rbrace$ & $n_7=7$ \\ \hline
$S_8=\left\lbrace 11,7,6\right\rbrace$ & $n_8=17$ \\ \hline
$S_9=\left\lbrace 10, 8 \right\rbrace$ & $n_9=10$ \\ \hline
$S_{10}=\left\lbrace 5\right\rbrace$ & $n_{10}=17$ \\ \hline
$S_{11}=\left\lbrace 3\right\rbrace$ & $n_{11}=10$ \\ \hline
$S_{12}=\left\lbrace 2\right\rbrace$ & $n_{12}=1$ \\ \hline
$S_{13}=\left\lbrace 1\right\rbrace$ & $n_{13}=7$ \\ \hline
\end{tabular}
\end{minipage}
\hspace{0.5cm}
\begin{minipage}[h]{0.45\linewidth}
\centering
\input{ExampleBijectionRuns}
\end{minipage}
\caption{Example of the bijection described in the proof of Theorem~\ref{Mappings/thm:bij_runs_partitions} for the mapping depicted in Figure~\ref{Mappings/fig:mapping_graph}.}
\label{Mappings/tab:ex_bij}
\end{figure}
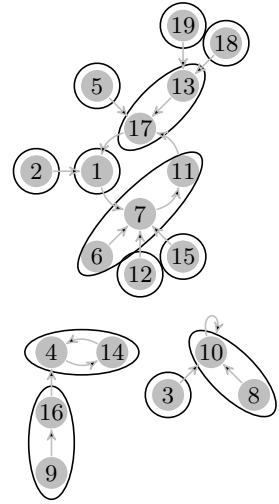

\begin{proof}[Proof]

First we remark that the statement of the theorem indeed will prove \eqref{eqn:Gnm_Formula}, 
since the number of set-partitions of $[n]$ into $m$ parts is given by $\Stir{n}{m}$ and the number of sequences $q$ satisfying the restrictions is given by $n \cdot (n-1) \cdots (n-m+1) = n^{\underline{m}}$.

To prove the theorem we consider an $n$-mapping with exactly $m$ runs and iterate the following procedure, where we colour the elements of the mapping until all elements are coloured.
\begin{itemize}
\item In the $j$-th step we consider the largest element in the mapping, which has not been coloured so far; let us denote it by $s_{1}^{(j)}$. Consider all preimages of $s_{1}^{(j)}$ with a label smaller than $s_{1}^{(j)}$ and, if there are such ones, take the one with largest label amongst them; let us denote this element by $s_{2}^{(j)}$. Then iterate this step with $s_{2}^{(j)}$, i.e., amongst all preimages of $s_{2}^{(j)}$ with a label smaller than $s_{2}^{(j)}$ take the one with largest label, which is denoted by $s_{3}^{(j)}$. After finitely many steps we arrive at an element $s_{k_{j}}^{(j)}$, which does not have preimages with a smaller label. We then define the set $S_{j} \colonequals \{s_{1}^{(j)}, \dots, s_{k_{j}}^{(j)}\}$. 
Note that in the mapping graph this corresponds to a path $s_{k_{j}}^{(j)} \to \dots \to s_{2}^{(j)} \to s_{1}^{(j)}$ with increasing labels on it. 
\item Additionally, we store in $n_{j}$ the image of $s_{1}^{(j)}$. 
Clearly $s_{1}^{(j)}$ is in $[n]$.
Due to the construction further restrictions hold:
Indeed, if $i < j$,  $n_{j}$ cannot be the smallest element in $S_{i}$ larger than $s_{1}^{(j)}$ (which, by construction, exists), since otherwise $s_{1}^{(j)}$ would have been chosen during the construction of the set $S_{i}$. 
\item 
Finally colour all elements of the mapping contained in $S_{j}$.
\end{itemize}

Since the mapping contains exactly $m$ runs and the smallest element in each set $S_{j}$ corresponds to the minimal element of a run, the procedure stops after exactly $m$ steps.
It thus defines a pair of a set partition $S=(S_{1}, \dots, S_{m})$ and a sequence $q=(n_{1}, \dots, n_{m})$ with the given restrictions. 

If the pair $(S,q)$ is given, the corresponding mapping can easily be reconstructed.
Indeed, the partition $S$ gives us a decomposition of the mapping into ascending runs and the sequence $q$ tells us how these runs have to be linked to each other.
The inverse of this bijection can therefore be defined in a straightforward way.
\end{proof}

\begin{example}
The construction of the partition $S$ and the sequence $q$ for the mapping described in Figure~\ref{Mappings/fig:mapping_graph} can be found in Figure~\ref{Mappings/tab:ex_bij}.
Let us exemplarily  explain how the set $S_8$ is constructed.
At this point, the elements in $\bigcup_{i=1}^7 S_i$, i.e., $19, 18, 17, 13, 16, 9, 15,14, 4$ and $12$, have already been coloured.
Thus, the largest element that has not been coloured so far is $11= s_1^{(8)}$.
For $s_2^{(8)}$, we consider the preimages of $11$ that have a label smaller than $11$. The only such element is $7$ and thus $s_2^{(8)}=7$.
Next, the preimages of $7$ are $6, 12$ and $15$ and  thus $s_3^{(8)}=6$.
Since $6$ does not have any preimages, we stop here and $\mathcal{S}_8= \left\lbrace 11,7,6 \right\rbrace$.
Since the image of $11$ is $17$, we set $n_8 =17$.
\demo\end{example}

\smallskip

Now we turn towards a combinatorial explanation of the direct link \eqref{eqn:GnmFnm_Relation} between ascending runs in labelled trees and mappings. In \cite{BruPan2016} the present authors generalized the concept of parking functions to labelled trees and mappings, and in this context they presented a bijection between parking functions on labelled trees and parking functions on mappings; the precise statement of this bijection can be found in Theorem 3.4 in~\cite{BruPan2016}. We adapt the idea of this bijective construction to obtain a bijection between marked trees and mappings, and thus gain (to the best of our knowledge) a new bijective proof of Cayley's formula. Moreover, as we will show after presenting this bijection, it preserves the number of ascending runs in the corresponding objects and thus provides the desired combinatorial proof of the statement.

In the following, we will denote by $t(x)$ the out-neighbour of node $x$ in the tree $t$.
That is, for $x$ a non-root node, $x \neq \textsf{root}(t)$, $t(x)$ is the unique node such that $(x, t(x))$ is an edge in $t$.
For the sake of convenience, let us define $t \left(\textsf{root}(t)\right) =\textsf{root}(t)$. First, we describe the bijection, and afterwards we show that the number of ascending runs will be preserved.

\begin{theorem}\label{thm:Mapping_Tree_Bijection}
  For each $n \ge 1$, there exists a bijection $\varphi$ from the set of pairs $(t,u)$, with $t$ a rooted labelled tree of size $n$ and $u \in t$ a node of $t$, to the set of $n$-mappings. Thus
\begin{equation*}
  n \cdot T_{n} = M_{n}, \quad \text{for $n \ge 1$}.
\end{equation*}
\end{theorem}

\begin{proof}
Given a pair $(t,u)$, we consider the unique path $u \rightsquigarrow \textsf{root}(t)$ from the node $u$ to the root of $t$.
It consists of the nodes $x_{1}=u$, $x_2=t(x_1), \ldots, x_{i+1}=t(x_i), \ldots, x_{r} = \textsf{root}(t)$, for some $r \geq 1$.
We denote by $I = (i_{1}, \dots, i_{k})$, with $i_{1} < i_{2} < \cdots < i_{k}$, for some $k \ge 1$, the indices of the right-to-left maxima in the sequence $x_1, x_2, \ldots, x_r$, i.e.,
\begin{equation*}
  i \in I \Longleftrightarrow x_{i} > x_{j}, \quad \text{for all $j > i$}.
\end{equation*}
The corresponding set of nodes in the path $u \rightsquigarrow \textsf{root}(t)$ will be denoted by $V_{I} \colonequals \{x_{i} : i \in I\}$. It follows from the definition that the root node is always contained in $V_{I}$, i.e., $x_{r} \in V_{I}$.

We can now describe the function $\varphi$ by constructing an $n$-mapping $f$.
The $k$ right-to-left maxima in the sequence $x_1, x_2, \ldots, x_r$ will give rise to $k$ connected components in the functional digraph $G_{f}$.
Moreover, the nodes on the path $u \rightsquigarrow \textsf{root}(t)$ in $t$ will correspond to the cyclic nodes in $G_{f}$. We describe $f$ by defining $f(x)$ for all $x \in [n]$, where we distinguish whether $x \in V_{I}$ or not.

\begin{itemize}
\item[$(a)$] Case $x \notin V_{I}$: We set $f(x) \colonequals t(x)$.
\item[$(b)$] Case $x \in V_{I}$: We set $f(x_{i_1}) \colonequals x_1$ and 
$f(x_{i_{j}}) \colonequals t\left(x_{i_{j-1}}\right)$, for $j >1$.

This means that the nodes on the path $u \rightsquigarrow \textsf{root}(t)$ in $t$ form $k$ cycles $C_{1} \colonequals (x_{1}, \dots, x_{i_{1}})$, \dots, $C_{k} \colonequals (t(x_{i_{k-1}}), \dots, x_{r}=x_{i_k})$ in $G_{f}$.
\end{itemize}

It is now easy to describe the inverse function $\varphi^{-1}$.
Given a mapping $f$, we sort the connected components of $G_f$ in decreasing order of their largest cyclic elements.
That is, if $G_f$ consists of $k$ connected components and $c_i$ denotes the largest cyclic element in the $i$-th component, we have $c_1 > c_2 > \ldots > c_k$.
Then, for every $1 \leq i \leq k$, we remove the edge $(c_i, d_i)$ where $d_i= f(c_i)$.
Next we reattach the components to each other by establishing the edges $(c_i, d_{i+1})$, for every $1 \leq i \leq k-1$.
This leads to the tree $t$.
Note that the node $c_k$ is attached nowhere since it constitutes the root of $t$.
Setting $u=d_1$, we obtain the preimage $(t,u)$ of $f$.
\end{proof}

\begin{example}
Taking the image of the pair $(t,1)$ where the tree $t$ is depicted in Figure~\ref{cayley/fig:example} leads to the mapping described in Figure~\ref{Mappings/fig:mapping_graph}.
We consider the unique path from the node labelled $1$ to the root of $t$. 
It consists of the following nodes: $1,7,11,17,4,14$ and $10$.
Within this sequence, the right-to-left maxima are $17, 14$ and $10$ which are marked by grey nodes in the figure.
When creating the image of $(t,1)$ under the map $\varphi$, the edges $(17,4)$ and $(14,10)$ are removed and the edges $(17,1)$, $(14,4)$ and $(10,10)$ are created.
\demo\end{example}

\begin{figure}
\centering
\input{BijectionCayleyTrees}
\caption{Taking the image of the pair $(t,1)$ where $t$ is depicted above leads to the mapping described in Figure~\ref{Mappings/fig:mapping_graph}. The unique path from $1$ to the root is marked by dashed edges.
Right-to-left-maxima on this path are marked by grey nodes.}
\label{cayley/fig:example}
\end{figure}
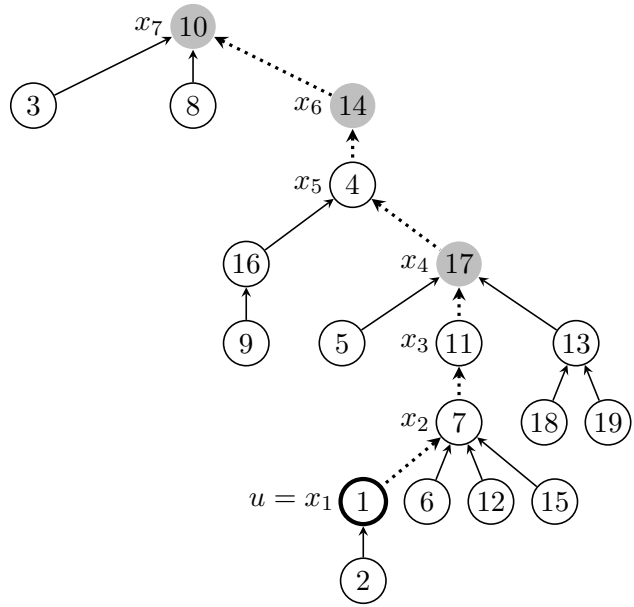

Next we show that $\varphi$ preserves the number of ascending runs, which relies on the following observation: When creating the image of some $(t,u)$ under the map $\varphi$, the only edges that are removed are descending ones. The edges that are created instead in $G_f$ are also descending (to be precise, non-ascending) ones. As a consequence, the property whether a node is (or is not) the starting node of an ascending run is preserved by $\varphi$.
\begin{theorem}
  The bijection $\varphi$ presented in Theorem~\ref{thm:Mapping_Tree_Bijection} preserves the number of ascending runs, i.e., for a pair $(t,u)$ of a labelled tree $t$ and a node $u \in t$, and the mapping $f = \varphi(t,u)$ being the image of $(t,u)$ under the map $\varphi$, it holds that $t$ and $f$ have the same number of ascending runs. Thus
	\begin{equation*}
	  \hat{G}_{n,m} = n \hat{F}_{n,m}.
	\end{equation*}
\end{theorem}
\begin{proof}
In order to prove the statement it suffices to show that in $t$ a node $x$ is not the starting node of an ascending run iff in $G_{f}$ node $x$ is not the starting node of an ascending run, i.e., that for every $x \in t$ it holds:
\begin{gather*}
  \text{In $t$ there exists a node $y < x$ with $t(y)=x$}\\ 
	\Updownarrow\\
	\text{In $f$ there exists an element $y < x$ with $f(y)=x$}.
\end{gather*}
As a consequence the starting nodes of an ascending run and thus the number of ascending runs in $t$ and $f$ must coincide.

To show this, we first assume that node $x$ in $t$ has an in-neighbour $y$ with $y < x$. Then $y$ cannot belong to the set of nodes $V_I$. Thus $f(y)=t(y)=x$ and in the mapping $f$ the node $x$ has a preimage $y$ with $y < x$.

For the other direction we assume that the node $x$ in $f$ has a preimage $y$ with $y<x$. Suppose $y=x_{i_{j}}$, for some $j$ in the construction of $\varphi$. Then either $x=f(y)=t(x_{i_{j-1}}) \leq y$, for $j > 1$, or $x=f(y)=x_{1} \le y$, for $j=1$, which in any case gives a contradiction. Thus $y \notin V_I$ and we have $f(y)=t(y)$. Therefore $y$ is an in-neighbour of $x$ in $t$ that satisfies $y < x$.
\end{proof}

\subsection{Distributional study}

Due to the relation $n^{n} \mathbb{P}\{Y_{n}^{[a]} = m\} = \hat{G}_{n,m} = n \hat{F}_{n,m} = n \cdot n^{n-1} \mathbb{P}\{X_{n}^{[a]} = m\}$ one immediately obtains that the probability mass function of the random variables $X_{n}^{[a]}$ and $Y_{n}^{[a]}$ counting the number of ascending runs in labelled trees of size $n$ and $n$-mappings, respectively, coincide; thus they are equally distributed, $X_{n}^{[a]} \stackrel{(d)}{=} Y_{n}^{[a]}$, and it suffices to only consider $X_{n}^{[a]}$ further. 

Of course, the explicit formula of $\hat{F}_{n,m}$ also characterizes the exact distribution of $X_{n}^{[a]}$; however, in order to get results for the moments or the limiting distribution of $X_{n}^{[a]}$ we prefer to consider the generating function $\hat{F}(z,v)$ characterized via \eqref{eqn:Fzv_Solution}. For the expectation we introduce $E_{1}(z) \colonequals \left.\frac{\partial}{\partial v} \hat{F}(z,v)\right|_{v=1} = \sum_{n \ge 1} n^{n-1} \mathbb{E}(X_{n}^{[a]}) \frac{z^{n}}{n!}$. Taking into account that $\hat{F}(z,1) = T(z)$, with $T=T(z)$ the tree function satisfying \eqref{eqn:TreeFunction}, and taking from \eqref{eqn:Fzv_Solution} the derivative w.r.t.\ $v$ easily shows that $E_{1}(z)$ is given as follows:
\begin{equation*}
  E_{1}(z) = \frac{e^{T}-1}{e^{T}(1-T)}.
\end{equation*}
An application of Cauchy's integral formula, where we use the functional equation \eqref{eqn:TreeFunction} of $T(z)$ and $T'(z) = \frac{e^{T}}{1-T}$, gives the exact and asymptotic result of $\mathbb{E}(X_{n}^{[a]})$ stated in Theorem~\ref{thm:XnaYna}:
\begin{align}
  \mathbb{E}(X_{n}^{[a]}) & = \frac{n!}{n^{n-1}} [z^{n}] E_{1}(z) = \frac{n!}{n^{n-1}} \frac{1}{2\pi i} \oint \frac{e^{(n+1)T(z)}}{T(z)^{n+1}} \frac{e^{T(z)}-1}{e^{T(z)}(1-T(z))} dz \notag\\
	& = \frac{n!}{n^{n-1}} \frac{1}{2\pi i} \oint \frac{e^{(n+1)T}}{T^{n+1}} \frac{e^{T}-1}{e^{T}(1-T)} \frac{1-T}{e^{T}} dT = \frac{n!}{n^{n-1}} [T^{n}] (e^{T}-1) e^{(n-1)T} \notag\\
	& = n - n (1-n^{-1})^{n} = (1-e^{-1}) n + \frac{e^{-1}}{2} + \mathcal{O}(n^{-1}). \label{eqn:Exp1Xna_asymptotic}
\end{align}

For the variance we consider $E_{2}(z) \colonequals \left.\frac{\partial^{2}}{\partial v^{2}} \hat{F}(z,v)\right|_{v=1} = \sum_{n \ge 1} n^{n-1} \mathbb{E}\left( X_{n}^{[a]} (X_{n}^{[a]}-1)\right) \frac{z^{n}}{n!}$, for which we get the following expression  after some computations:
\begin{equation*}
  E_{2}(z) = \frac{(1-e^{2T}) T^{2} + (1-4e^{T}+3e^{2T}) T - 1 + 2e^{T} -e^{2T}}{e^{2T} (1-T)^{3}}.
\end{equation*}
To get the asymptotic behaviour of the coefficients, and thus of the second factorial moment of $X_{n}^{[a]}$, we use a basic application of singularity analysis, where we require the local behaviour of $T(z)$ in a complex neighbourhood of the dominant singularity $z=e^{-1}$ (which is also the dominant singularity of $E_{2}(z)$), which can be found, e.g., in \cite{FlaSed2009}:
\begin{equation}\label{eqn:Tz_LocalExpansion}
  T(z) = 1- \sqrt{2} \sqrt{1-ez} + \frac{2}{3} (1-ez) - \frac{11 \sqrt{2}}{36} (1-ez)^{\frac{3}{2}} + \mathcal{O}\left( (1-ez)^{2}\right).
\end{equation}
This gives the following local expansion of $E_{2}(z)$ around $z=e^{-1}$:
\begin{equation*}
  E_{2}(z) = \frac{\sqrt{2} \, (1-e^{-1})^{2}}{4 (1-ez)^{\frac{3}{2}}} - \frac{\sqrt{2} \, (35-94e^{-1}+83e^{-2})}{48 \sqrt{1-ez}} + \mathcal{O}(1),
\end{equation*}
and, after an application of transfer lemmata, the following asymptotic expansion of the coefficients, for $n \to \infty$:
\begin{equation}\label{eqn:Exp2Xna_asymptotic}
  \mathbb{E}\left( X_{n}^{[a]} (X_{n}^{[a]}-1)\right) = (1-e^{-1})^{2} n^{2} + (-1+3e^{-1}-3e^{-2}) n + \mathcal{O}(1).
\end{equation}
The result for the variance given in Theorem~\ref{thm:XnaYna} easily follows from \eqref{eqn:Exp1Xna_asymptotic} and \eqref{eqn:Exp2Xna_asymptotic}:
\begin{equation}\label{eqn:VarXna_asymptotic}
  \mathbb{V}(X_{n}^{[a]}) = \mathbb{E}\left( X_{n}^{[a]} (X_{n}^{[a]}-1)\right) + \mathbb{E}(X_{n}^{[a]}) - \mathbb{E}(X_{n}^{[a]})^{2}
	= (e^{-1} - 2 e^{-2}) n + \mathcal{O}(1).
\end{equation}

Moreover, when studying the function $\hat{G}(z,v)$ in a complex neighbourhood of $v=1$ and applying the so-called quasi-power theorem of Hwang \cite{Hwa1998} one can deduce the central limit theorem for $X_{n}^{[a]}$ stated in Theorem~\ref{thm:XnaYna}. However, in Section~\ref{ssec:JointLimitingDistribution} we will study the limiting behaviour of the joint distribution of the number of ascending runs $X_{n}^{[a]}$ and descending runs $X_{n}^{[d]}$, from which the central limit theorem for the marginal variables follows as well; thus we omit such computations here.

\section{Joint behaviour of ascending and descending runs}\label{sec:JointStudies}

\subsection{Limiting distribution results}\label{ssec:JointLimitingDistribution}

In order to show, after normalization, convergence in distribution of the random vector $\bm{X}_{n} \colonequals \left(\begin{smallmatrix}X_{n}^{[a]}\\ X_{n}^{[d]}\end{smallmatrix}\right)$ of the number of ascending and descending runs in random labelled trees to a bivariate normal distribution, we will study the asymptotic behaviour of the bivariate moment generating function $\mathbb{E}(e^{X_{n}^{[a]} s_{1} + X_{n}^{[d]} s_{2}})$ in a complex neighbourhood of $(s_{1},s_{2})=(0,0)$ and apply a bivariate extension of the already mentioned quasi power-theorem, which is due to Heuberger \cite{Heu2007}. Actually, an additional contribution of this theorem is to provide bounds on the rate of convergence to the limiting distribution, which thus also hold for $\bm{X}_{n}$, but we decided to omit such results here.

To obtain the asymptotic behaviour of the aforementioned bivariate moment generating function, we will use the concept of singularity perturbation analysis, see \cite{FlaSed2009}, by studying the local behaviour around the dominant singularity of the generating function $F(z,v,w)$ defined via the functional equation \eqref{eqn:Fzvw_Solution}, where one considers $v$, $w$ as fixed parameters chosen in a complex neighbourhood of $1$. When we consider the defining equation of $F$ for $v=w=1$, we obtain the functional equation of the tree function $T(z)$, $z = \frac{F}{e^{F}}$. We quickly recapitulate the considerations yielding the analytic behaviour of this function $F=F(z,1,1)=T(z)$, see \cite{FlaSed2009}. According to the implicit function theorem, when defining $h(F,z) \colonequals \frac{F}{e^{F}}-z$, the equation $h(F,z)=0$ cannot be resolved w.r.t.\ $F$ locally in a unique way for points $(F,z) = (\tau,\rho)$ satisfying
\begin{equation}\label{eqn:htaurho_Defining}
  h(\tau, \rho)=0 \quad \text{and} \quad \left.\frac{\partial}{\partial F}h(F,z)\right|_{(F,z)=(\tau,\rho)}=0, 
\end{equation}
yielding the unique solution $\tau=1$ and $\rho=e^{-1}$; $z=\rho$ is the dominant singularity (a branch point) of $F$ whose local expansion around $\rho$ is given by \eqref{eqn:Tz_LocalExpansion}. Now we consider the function $F=F(z,v,w)$ for $v$, $w$ close to $1$, and thus define
\begin{equation*}
  h(F,z) = \frac{\ln\left(\frac{(e^{F}-1+v)(e^{F}-1+w)}{vwe^{F}}\right)}{e^{F}-(1-v)(1-w)} -z.
\end{equation*}
Analogously, the equation $h(F,z)=0$ cannot be resolved w.r.t.\ $F$ locally in a unique way for points $(F,z) =(\tau,\rho)$, with $\tau=\tau(v,w)$ and $\rho=\rho(v,w)$, satisfying equation \eqref{eqn:htaurho_Defining}, which characterizes $\tau$ as solution of the equation
\begin{multline}
  g(\tau, v, w) := \ln\left(\frac{(e^{\tau}-1+v)(e^{\tau}-1+w)}{vwe^{\tau}}\right)\\
	\mbox{} - \frac{e^{3\tau} - (1-v)(1-w) e^{2\tau} -(1-v)(1-w)e^{\tau} +(1-v)^{2}(1-w)^{2}}{(e^{\tau}-1+v)(e^{\tau}-1+w)e^{\tau}} = 0,\label{eqn:tau_Def}
\end{multline}
and the dominant singularity $z=\rho$ of $F$ is given via
\begin{equation}\label{eqn:rho_Def}
  \rho = \frac{e^{3\tau} - (1-v)(1-w) e^{2\tau} -(1-v)(1-w)e^{\tau} +(1-v)^{2}(1-w)^{2}}{\left( e^{\tau}-(1-v)(1-w)\right)(e^{\tau}-1+v)(e^{\tau}-1+w)e^{\tau}}.
\end{equation}
Note that equation \eqref{eqn:tau_Def} has for $v=w=1$ the unique solution $\tau=1$. Since the function $g$ is analytic around $v=w=1$ and $g_{\tau}(\tau,1,1)=1 \neq 0$ as can be checked easily, another application of the analytic implicit function theorem guarantees that there is a uniquely determined analytic function $\tau(v,w)$ around $v=w=1$ satisfying \eqref{eqn:tau_Def}. Due to \eqref{eqn:rho_Def} this also shows that $\rho(v,w)$ is an analytic function around $v=w=1$. A series expansion of the functional equation \eqref{eqn:Fzvw_Solution} around $F=\tau$ and $z=\rho$ gives after some computations the following local expansion of $F(z,v,w)$ around $z=\rho(v,w)$:
\begin{equation}\label{eqn:Fzvw_LocalExpansion}
  F = \tau - \sqrt{\kappa} \sqrt{1-\frac{z}{\rho}} + \mathcal{O}\left( 1-\frac{z}{\rho}\right),
\end{equation}
with $\kappa=\kappa(v,w)$ given as follows, where we use the abbreviations $\bar{v}=1-v$ and $\bar{w}=1-w$:
\begin{equation*}
  \kappa = \frac{2(e^{\tau}-\bar{v})(e^{\tau}-\bar{w})\left( e^{3\tau}-\bar{v}\bar{w}e^{2\tau}-\bar{v}\bar{w}e^{\tau}+\bar{v}^{2}\bar{w}^{2}\right)}{e^{\tau} \left( e^{4\tau}-4\bar{v}\bar{w}e^{2\tau}+2\bar{v}\bar{w}(\bar{v}+\bar{w})e^{\tau}-\bar{v}^{2}\bar{w}^{2}\right)}.
\end{equation*}
An application of singularity analysis to \eqref{eqn:Fzvw_LocalExpansion} then shows  the following asymptotic behaviour of the coefficients of $F(z,v,w)$, and thus of the probability generating function of the random vector $\bm{X}_{n}$:
\begin{equation*}
  \mathbb{E}\left( v^{X_{n}^{[a]}} w^{X_{n}^{[d]}}\right) = \frac{n!}{n^{n-1}} [z^{n}] F(z,v,w) = {\textstyle{\sqrt{\frac{\kappa}{2}}}} \cdot \frac{1}{(e \rho)^{n}} \cdot \left( 1+\mathcal{O}(n^{-1})\right).
\end{equation*}
Setting $v=e^{s_{1}}$ and $w=e^{s_{2}}$, we obtain the required asymptotic expansion of the bivariate moment generating function:
\begin{equation}\label{eqn:XnaXnd_MGF}
  \mathbb{E}\left( e^{X_{n}^{[a]} s_{1} + X_{n}^{[d]} s_{2}}\right) = e^{U(s_{1},s_{2}) \cdot n + V(s_{1},s_{2})} \cdot \left( 1+\mathcal{O}(n^{-1})\right),
\end{equation}
with functions $U$ and $V$ given as follows:
\begin{equation*}
  U(s_{1},s_{2}) = -\left( 1+\ln\left( \rho(e^{s_{1}},e^{s_{2}})\right)\right), \qquad V(s_{1},s_{2}) = \frac{1}{2} \ln\left( \frac{\kappa(e^{s_{1}},e^{s_{2}})}{2}\right).
\end{equation*}
This is exactly the setting of the bivariate quasi-power theorem due to Heuberger (see \cite{Heu2007}): Under the assumption that $U$ and $V$ are analytic around $s_{1}$, $s_{2}$ and the Hessian matrix $H_{U}(0,0)$ of $U$ evaluated at $s_{1}=s_{2}=0$ is invertible (which is satisfied here), we obtain from \eqref{eqn:XnaXnd_MGF} that
\begin{equation*}
  \mathbb{E}(\bm{X}_{n}) \sim \text{grad} \, U(0,0) \cdot n
\end{equation*}
and
\begin{equation*}
  \frac{1}{\sqrt{n}}\left( \bm{X}_{n} - \text{grad} \, U(0,0) \cdot n\right) \xrightarrow{(d)} \mathcal{N}(\bm{0},\bm{\Sigma}),
\end{equation*}
with $\bm{\Sigma} = H_{U}(0,0)$.
Due to the following local expansion of $U(s_{1},s_{2})$ around $s_{1}=s_{2}=0$:
\begin{align*}
  U(s_{1},s_{2}) & = (1-e^{-1}) s_{1} + (1-e^{-1}) s_{2} \\
	& \quad \mbox{} + \frac{1}{2}(e^{-1}-2e^{-2}) s_{1}^{2} + (e^{-1}-3e^{-2})s_{1}s_{2} + \frac{1}{2}(e^{-1}-2e^{-2}) s_{2}^{2} + \mathcal{O}\left( \|(s_{1},s_{2})\|^{3}\right),
\end{align*}
we obtain the bivariate limiting distribution result for runs in labelled trees $\bm{X}_{n}$ stated in Theorem~\ref{thm:XnYn}. Note that due to $\frac{e^{-1} - 3 e^{-2}}{e^{-1}-2e^{-2}} = -0.3922\ldots$ the r.v.\ $X_{n}^{[a]}$ and $X_{n}^{[d]}$ are negatively correlated, which is in accordance with the intuition.

\smallskip

Due to equation \eqref{eqn:Gzvw_Solution} and \eqref{eqn:FzvwGzvw_Connection} connecting the functions $F(z,v,w)$ and $G(z,v,w)$, the corresponding result for runs in mappings $\bm{Y}_{n} = \left(\begin{smallmatrix}Y_{n}^{[a]} \\ Y_{n}^{[d]}\end{smallmatrix}\right)$ can be obtained in a rather straightforward way. Namely, the dominant singularity of $G(z,v,w)$, for $v$ and $w$ in a neighbourhood of $v=w=1$, is also given at $z=\rho(v,w)$, which follows from \eqref{eqn:FzvwGzvw_Connection}, or alternatively from the fact that the denominator of \eqref{eqn:Gzvw_Solution} vanishes for $F=\tau(v,w)$. Using the expansion \eqref{eqn:Fzvw_LocalExpansion} for $F$ around $z=\rho$ we obtain the following local expansion of $G(z,v,w)$ around $z=\rho(v,w)$:
\begin{equation}\label{eqn:Gzvw_LocalExpansion}
  G(z,v,w) = \frac{1}{\sqrt{\chi} \sqrt{1-\frac{z}{\rho}}} + \mathcal{O}(1),
\end{equation}
with $\chi=\chi(v,w)$ given as follows (again using the abbreviations $\bar{v}=1-v$ and $\bar{w}=1-w$):
\begin{equation*}
  \chi = \frac{2 (e^{3\tau}-\bar{v}\bar{w}e^{2\tau}-\bar{v}\bar{w}e^{\tau}+\bar{v}^{2}\bar{w}^{2})(e^{4\tau}-4\bar{v}\bar{w}e^{2\tau}+2\bar{v}\bar{w}(\bar{v}+\bar{w})e^{\tau}-\bar{v}^{2}\bar{w}^{2})}{e^{3\tau}(e^{\tau}-\bar{v})(e^{\tau}-\bar{w})(e^{\tau}-\bar{v}\bar{w})^{2}}.
\end{equation*}
Singularity analysis applied to \eqref{eqn:Gzvw_LocalExpansion} gives the following asymptotic expansion of the probability generating function of $\bm{Y}_{n}$,
\begin{equation*}
  \mathbb{E}\left( v^{Y_{n}^{[a]}} w^{Y_{n}^{[d]}}\right) = \frac{n!}{n^{n}} [z^{n}] G(z,v,w) = {\textstyle{\sqrt{\frac{2}{\chi}}}} \cdot \frac{1}{(e \rho)^{n}} \cdot \left( 1+\mathcal{O}(n^{-1})\right),
\end{equation*}
and thus, by setting $v=e^{s_{1}}$ and $w=e^{s_{2}}$, of the moment generating function:
\begin{equation}\label{eqn:YnaYnd_MGF}
  \mathbb{E}(e^{Y_{n}^{[a]} s_{1} + Y_{n}^{[d]} s_{2}}) = e^{U(s_{1},s_{2}) \cdot n + \tilde{V}(s_{1},s_{2})} \cdot \left( 1+\mathcal{O}(n^{-1})\right),
\end{equation}
with $U(s_{1},s_{2})$ appearing in \eqref{eqn:XnaXnd_MGF} and $\tilde{V}(s_{1},s_{2}) = -\frac{1}{2}\ln\left(\frac{2}{\chi(e^{s_{1}},e^{s_{2}})}\right)$.
Thus, from the bivariate quasi-power theorem \cite{Heu2007} we deduce that $\bm{X}_{n}$ and $\bm{Y}_{n}$ have the same limiting behaviour, which is stated in Theorem~\ref{thm:XnYn}.

\subsection{Relation to a joint study of ascents and leaves by Gessel}

Gessel \cite{Ges1996} gave a joint study of descents and leaves in forests of rooted labelled trees. Of course, due to symmetry, all enumeration results also hold for a joint study of ascents and leaves in these structures, where an ascent is defined as a node which has at least one in-neighbour with a smaller label. It was already pointed out in \cite{Ges1996} that the results given there could be transferred easily to rooted labelled trees (instead of forests). Conversely, all results obtained in the present work easily give corresponding results for forests of trees. In particular, the limiting distribution results also hold for forests. Since we focus on trees here, we reformulate a main result of \cite{Ges1996}, which concerns the characterization of the generating function jointly counting ascents and leaves: let $A_{n,m,\ell}$ be the number of size-$n$ trees with $m$ ascents and $\ell$ leaves and $A(z,v,w) \colonequals \sum_{n \ge 1} \sum_{m \ge 0} \sum_{\ell \ge 1} A_{n,m,\ell} \, \frac{z^{n} v^{m} w^{\ell-1}}{n!}$ its generating function; then $A \colonequals A(z,v,w)$ is characterized as solution of the functional equation
\begin{equation}\label{eqn:Azvw_Solution}
  z = \frac{\ln\left(\frac{(ve^{A} -v+1)(we^{A}-w+1)}{e^{A}}\right)}{vwe^{A}-(1-v)(1-w)}.
\end{equation}
Since $A(z,v,w)$ is symmetric in $v$ and $w$, this has the interesting consequence that the number of labelled trees of a certain size with $a$ ascents and $b+1$ leaves equals the number of trees of the corresponding size with $b$ ascents and $a+1$ leaves. 
Actually, a proof of this fact motivated the study of $A(z,v,w)$ in \cite{Ges1996} and combinatorial explanations of this symmetry relation can be found in \cite{Kal2002}.

When considering $F$ and $A$ characterized by the functional equations \eqref{eqn:Fzvw_Solution} and \eqref{eqn:Azvw_Solution}, respectively, one obtains that they are related via $F(z,v,w) = A(zvw,\frac{1}{v},\frac{1}{w})$. At the level of coefficients this yields $A_{n,m,\ell} = F_{n,n-m,n+1-\ell}$, i.e., that the number of trees of size $n$ with $m$ ascents and $\ell$ leaves is equal to the number of trees of the same size with $n-m$ ascending runs and $n+1-\ell$ descending runs. Due to the characterization of the starting node of an ascending run as a node without having an in-neighbour with a smaller label, i.e., as a node that is not an ascent, it even follows for any fixed tree $t$ that it has $m$ ascents iff it has $|t|-m$ ascending runs. However, for the relation on the joint behaviour of ascending and descending runs, and ascents and leaves, respectively, we do not have a combinatorial explanation.
Nevertheless, as a consequence of this connection one can easily deduce from Theorem~\ref{thm:XnYn} bivariate distribution results for the number of ascents and leaves in random labelled trees, which we want to state in the following.
\begin{theorem}
Let $\bm{A}_{n} \colonequals \left(\begin{smallmatrix}A_{n}^{[A]}\\ A_{n}^{[L]}\end{smallmatrix}\right)$ be the random vector jointly counting the number of ascents $A_{n}^{[A]}$ and the number of leaves $A_{n}^{[L]}$ in a random size-$n$ tree. Then it holds that 
$\bm{A}_{n} \stackrel{(d)}{=} \left(\begin{smallmatrix}n\\ n+1 \end{smallmatrix}\right) - \bm{X}_{n}$, with $\bm{X}_{n}$ the random vector counting ascending and descending runs in labelled trees as introduced in Theorem~\ref{thm:XnYn}. Moreover, after suitable normalization, $\bm{A}_{n}$ converges in distribution to a bivariate normal distribution:
\begin{equation*}
  \frac{1}{\sqrt{n}} \left(\bm{A}_{n} - \left(\begin{smallmatrix}e^{-1}\\ e^{-1}\end{smallmatrix}\right) \cdot n\right) \xrightarrow{(d)} \mathcal{N}(\bm{0}, \bm{\Sigma}),
\end{equation*}
where the variance-covariance matrix $\bm{\Sigma}$ is stated in Theorem~\ref{thm:XnYn}.
\end{theorem}

\end{document}

%% file: FunctionalGraphGrey.tex
\begin{tikzpicture}[->,>=stealth',
auto,node distance=1.3cm,
                    semithick, scale=0.5]
  \tikzstyle{every state}=[draw=black,circle,fill=none,minimum size=17pt,inner sep=0pt]
	\tikzstyle{bluereds}=[draw=black,circle split,circle split part fill={black,black!40},line width=0mm,minimum size=17pt,inner sep=0pt]
	\tikzstyle{blues}=[draw=black,circle split,circle split part fill={black,white},line width=0mm,minimum size=17pt,inner sep=0pt]
	\tikzstyle{reds}=[draw=black,circle split,circle split part fill={white,black!40},line width=0mm,minimum size=17pt,inner sep=0pt]
	\tikzstyle{whites}=[draw=black,circle split,circle split part fill={white,white},line width=0mm,minimum size=17pt,inner sep=0pt]

  \node[state,blues]			(1)                    {};
  \node[state,whites]         (7) [below  right of=1] {};
  \node[state,reds]         (11) [above right of=7] {};
  \node[state,reds]         (17) [above left of=11] {};
  \node[state,bluereds]         (5) [above left of=17]       {};
  \node[state,bluereds]         (2) [left of=1]       {};
  \node[state,bluereds]         (6) [below left of=7]       {};
  \node[state,bluereds]         (12) [below of=7]       {};
  \node[state,bluereds]         (15) [below right of=7]       {};
  \node[state,blues]         (13) [above right of=17]       {};
  \node[state,bluereds]         (19) [above of=13]       {};
  \node[state,bluereds]         (18) [above right of=13]       {};
  \node[state]			(1label)                    {\textcolor{gray}{$\bm{1}$}};
  \node[state]         (7label) [below  right of=1] {$\bm{7}$};
  \node[state]         (11label) [above right of=7] {$\bm{11}$};
  \node[state]         (17label) [above left of=11] {$\bm{17}$};
  \node[state]         (5label) [above left of=17]       {\textcolor{white}{$\bm{5}$}};
  \node[state]         (2label) [left of=1]       {\textcolor{white}{$\bm{2}$}};
  \node[state]         (6label) [below left of=7]       {\textcolor{white}{$\bm{6}$}};
  \node[state]         (12label) [below of=7]       {\textcolor{white}{$\bm{12}$}};
  \node[state]         (15label) [below right of=7]       {\textcolor{white}{$\bm{15}$}};
  \node[state]         (13label) [above right of=17]       {\textcolor{gray}{$\bm{13}$}};
  \node[state]         (19label) [above of=13]       {\textcolor{white}{$\bm{19}$}};
  \node[state]         (18label) [above right of=13]       {\textcolor{white}{$\bm{18}$}};

  \path  	(2) edge (1)
				(1) edge [bend right]  (7) 
				(7) edge [bend right]   (11) 
				(11) edge [bend right]   (17) 
				(5) edge   (17)
				(6) edge (7)
				(12) edge (7)
				(15) edge (7)
				(19) edge (13)
				(18) edge (13)
				(13) edge (17)
				(17) edge [bend right]   (1) ;

\begin{scope}[xshift=-14cm,yshift=2cm]
  \node[state,reds]			(10)                    {};
  \node[state,bluereds]         (3) [below  left of=10] {};
  \node[state,bluereds]         (8) [below right of=10] {};
  \node[state]			(10label)                    {$\bm{10}$};
  \node[state]         (3label) [below  left of=10] {\textcolor{white}{$\bm{3}$}};
  \node[state]         (8label) [below right of=10] {\textcolor{white}{$\bm{8}$}};

  \path  	(8) edge (10)
				(3) edge (10)
				(10) edge [loop above] (10) ;
\end{scope}

\begin{scope}[xshift=-8cm, yshift=2cm]
  \node[state,blues]			(4)                    {};
  \node[state,reds]         (14) [right of=4] {};
  \node[state,reds]         (16) [below of=4] {};
  \node[state,bluereds]         (9) [below of=16] {};
  \node[state]			(4)                    {\textcolor{gray}{$\bm{4}$}};
  \node[state]         (14) [right of=4] {$\bm{14}$};
  \node[state]         (16) [below of=4] {$\bm{16}$};
  \node[state]         (9) [below of=16] {\textcolor{white}{$\bm{9}$}};

  \path  	(9) edge (16)
				(16) edge (4)
				(4) edge [bend right] (14) 
				(14) edge [bend right] (4) ;
\end{scope}

\end{tikzpicture}

%% file: ExampleBijectionRuns.tex
\usetikzlibrary{automata, arrows,calc}
\begin{tikzpicture}[->,>=stealth',
draw=black!25,auto,node distance=0.8cm,
                    semithick, scale=0.3]
  \tikzstyle{every state}=[font=\footnotesize,draw=none,circle,fill=black!25,minimum size=12pt,inner sep=0pt]]

  \node[state]			(1)                    {$1$};
  \node[state]         (7) [below  right of=1] {$7$};
  \node[state]         (11) [above right of=7] {$11$};
  \node[state]         (17) [above left of=11] {$17$};
  \node[state]         (5) [above left of=17]       {$5$};
  \node[state]         (2) [left of=1]       {$2$};
  \node[state]         (6) [below left of=7]       {$6$};
  \node[state]         (12) [below of=7]       {$12$};
  \node[state]         (15) [below right of=7]       {$15$};
  \node[state]         (13) [above right of=17]       {$13$};
  \node[state]         (15) [below right of=7]       {$15$};
  \node[state]         (19) [above of=13]       {$19$};
  \node[state]         (18) [above right of=13]       {$18$};

\node (1713) at ($(17)!0.5!(13)$) {};

\draw[black,rotate=45] (7) ellipse (3.7cm and 1.1cm);
\draw[black,rotate=45] (19) ellipse (1cm and 1cm);
\draw[black,rotate=45] (18) ellipse (1cm and 1cm);
\draw[black,rotate=45] (1713) ellipse (2.5cm and 1cm);
\draw[black,rotate=45] (15) ellipse (1cm and 1cm);
\draw[black,rotate=45] (12) ellipse (1cm and 1cm);
\draw[black,rotate=45] (5) ellipse (1cm and 1cm);
\draw[black,rotate=45] (2) ellipse (1cm and 1cm);
\draw[black,rotate=45] (1) ellipse (1cm and 1cm);

  \path  	(2) edge (1)
				(1) edge [bend right]  (7) 
				(7) edge [bend right]   (11) 
				(11) edge [bend right]   (17) 
				(5) edge   (17)
				(6) edge (7)
				(12) edge (7)
				(15) edge (7)
				(19) edge (13)
				(18) edge (13)
				(13) edge (17)
				(17) edge [bend right]   (1) ;

\begin{scope}[xshift=5cm,yshift=-8cm]
  \node[state]			(10)                    {$10$};
  \node[state]         (3) [below  left of=10] {$3$};
  \node[state]         (8) [below right of=10] {$8$};

\node (108) at ($(10)!0.5!(8)$) {};
\draw[black,rotate=45] (3) ellipse (1cm and 1cm);
\draw[black,rotate=-45] (108) ellipse (2.5cm and 1cm);

  \path  	(8) edge (10)
				(3) edge (10)
				(10) edge [loop above] (10) ;
\end{scope}

\begin{scope}[xshift=-2cm, yshift=-8cm]
  \node[state]			(4)                    {$4$};
  \node[state]         (14) [right of=4] {$14$};
  \node[state]         (16) [below of=4] {$16$};
\node[state]         (9) [below of=16] {$9$};

\node (169) at ($(16)!0.5!(9)$) {};
\node (144) at ($(14)!0.5!(4)$) {};
\draw[black,rotate=90] (169) ellipse (2.5cm and 1cm);
\draw[black] (144) ellipse (2.5cm and 1cm);

  \path  	(9) edge (16)
				(16) edge (4)
				(4) edge [bend right] (14) 
				(14) edge [bend right] (4) ;
\end{scope}

\end{tikzpicture}

%% file: BijectionCayleyTrees.tex
\begin{tikzpicture}[scale=0.7,grow=down, stealth-, semithick]
\tikzstyle{s'}=[circle,fill=black!25,minimum size=17pt,inner sep=0pt, solid]] 
\tikzstyle{level 1}=[ sibling distance=3cm]  
\tikzstyle{level 3}=[ sibling distance=4cm]  
\tikzstyle{level 4}=[ sibling distance=2.2cm]  
\tikzstyle{level 5}=[ sibling distance=1.2cm]  
 \tikzstyle{s}=[draw=black,circle,fill=none,minimum size=17pt,inner sep=0pt, solid,semithick]] 
 \tikzstyle{emph}=[edge from parent/.style={black,very thick,draw, dotted}]
 \tikzstyle{norm}=[edge from parent/.style={black, semithick, draw,solid}]

\node[s'] (root) {$10$} 
child { node[s]  {$3$}}
child { node[s] {$8$}}
child[emph] { node[s'] (14) {$14$}
	child {node[s] (4) {$4$}
			child[norm] {node[s] (16) {$16$}
					child {node[s] {$9$}}
			}
			child[emph] {node[s'] (17) {$17$}
					child[norm] { node[s] {$5$}}
					child { node[s] (11) {$11$}
							child {node[s] (7) {$7$}
									child[emph] {node[s, ultra thick] (1) {$1$}
											child[norm] {node[s] {$2$}}
									}
									child[norm] {node[s] {$6$}}
									child[norm] {node[s] {$12$}}
									child[norm] {node[s] {$15$}}
							}
					}
					child[norm,dashed] { node[s] {$13$}
							child {node[s] (18) {$18$}}
							child {node[s] (19) {$19$}}
					}
			}
	}
}
;
\begin{scope}[nodes = {left = 7pt}]
 \node  at (1) {$u=x_1$};
\node  at (7) {$x_2$};
\node  at (11) {$x_3$};
 \node  at (17) {$x_4$};
 \node  at (4) {$x_5$};
 \node  at (14) {$x_6$};
 \node  at (root) {$x_7$};
\end{scope}
\end{tikzpicture}